\documentclass{gtmon_a}
\pdfoutput=1

\proceedingstitle{Proceedings of the School and Conference in Algebraic
Topology (The Vietnam National University, Hanoi, 9-20 August 2004)}
\conferencestart{9 August 2004}
\conferenceend{20 August 2004}
\conferencename{School and Conference in Algebraic Topology}
\conferencelocation{Vietnam National University, Hanoi, Vietnam}

\editor{John Hubbuck}
\givenname{John}
\surname{Hubbuck}

\editor{Nguy\~\ecircumflex{}n H V H\uhorn{}ng}
\givenname{H\uhorn{}ng}
\surname{Nguy\~\ecircumflex{}n}

\editor{Lionel Schwartz}
\givenname{Lionel}
\surname{Schwartz}

\title[Cohomology of $p$--local  groups over $p_+^{1+2}$]
{Stable splitting and cohomology of 
$p$--local finite groups over\\the extraspecial $p$--group
of order $p^3$ and exponent $p$}
\author{Nobuaki Yagita}
\givenname{Nobuaki}
\surname{Yagita}
\address{Department of Mathematics\\
Faculty of Education\\
Ibaraki University\\\newline
Mito\\
Ibaraki\\
Japan}
\email{yagita@mx.ibaraki.ac.jp}
\urladdr{}

\volumenumber{11}
\issuenumber{}
\publicationyear{2007}
\papernumber{19}
\startpage{399}
\endpage{434}

\doi{}
\MR{}
\Zbl{}

\arxivreference{}




\received{28 December 2004}
\revised{17 July 2005}
\accepted{18 December 2005}
\published{14 November 2007}
\publishedonline{14 November 2007}
\proposed{}
\seconded{}
\corresponding{}
\editor{}
\version{}

\usepackage{rotating}

\renewcommand{\theenumi}{\arabic{enumi}}
\newcommand{\odd}{\mathrm{odd}}
\newcommand{\even}{\mathrm{even}}
\makeop{Out}
\makeop{Ker}
\makeop{Aut}
\newcommand{\ddet}{\operatorname{Det}}
\makeop{rank}
\makeop{pr}
\makeop{Hom}
\newcommand{\rad}{\operatorname{radical}}
\makeop{diag}
\newcommand{\mmod}{\,\operatorname{mod}}
\numberwithin{equation}{section}
\newcommand{\ec}{\mathrm{ec}}

\newcommand{\bZ}{{\mathbb Z}}
\newcommand{\bF}{{\mathbb F}}

\newcommand{\bM}{{\mathbb M}}

\makeatletter
\def\cnewtheorem#1[#2]#3{\newtheorem{#1}{#3}[section]
\expandafter\let\csname c@#1\endcsname\c@thm}

\newtheorem{thm}{Theorem}[section]
\cnewtheorem{lemma}[thm]{Lemma}
\cnewtheorem{cor}[thm]{Corollary}
\cnewtheorem{prop}[thm]{Proposition}
\theoremstyle{remark}
\newtheorem*{rem}{Remark}

\makeatother  %  move after \newtheorem block

\begin{document}

\begin{asciiabstract} 
Let p be an odd prime.  Let G be a p-local finite group 
over the extraspecial p-group p_+^{1+2}.  
In this paper we study
the cohomology and the stable splitting of their p-complete
classifying space BG.
\end{asciiabstract}

\begin{htmlabstract}
Let p be an odd prime.  Let G be a p&ndash;local finite group
over the extraspecial p&ndash;group p<sub>+</sub><sup>1+2</sup>.
In this paper we study
the cohomology and the stable splitting of their  p&ndash;complete
classifying space BG.
\end{htmlabstract}

\begin{abstract} 
Let $p$ be an odd prime.  Let $G$ be a $p$--local finite group 
over the extraspecial $p$--group $p_+^{1+2}$.  
In this paper we study
the cohomology and the stable splitting of their  $p$--complete
classifying space $BG$.
\end{abstract}

\maketitle

%%% Document itself goes here.

\section{Introduction}
\label{sec:sec1}

Let us write by $E$ the extraspecial $p$--group $p_+^{1+2}$ 
of order $p$ and exponent $p$ for an odd prime $p$.  Let $G$ be a 
finite group
having $E$ as a $p$--Sylow subgroup, and  $BG$ ($=BG_p^{\wedge}$) the
$p$--completed classifying space of $G$.  In papers by Tezuka and Yagita
\cite{T-Y} and Yagita \cite{Y1,Y2}, the 
cohomology and 
stable splitting for such groups are studied. 
In many cases non isomorphic groups have homotopy equivalent 
$p$--completed 
classifying spaces, showing that there are not too many homotopy 
types of $BG$, as 
was first suggested by  C\,B Thomas \cite{Th} and D Green \cite{G}.

Recently, Ruiz and Viruel \cite{R-V} classified all $p$--local finite groups
for the $p$--group $E$.
Their results show that each classifying space $BG$ is homotopic to 
one of the classifying spaces which were studied in \cite{T-Y} or
classifying spaces of three exotic $7$--local finite groups.
(While descriptions in \cite{T-Y} of $H^*(^2F_4(2)')_{(3)}$ 
$H^*(Fi_{24}')_{(7)}$  and $H^*(\bM)_{(13)}$ contained some errors.)

In \fullref{sec:sec2}, we recall the results of Ruiz and Viruel.
In \fullref{sec:sec3}, we also recall the cohomology $H^*(BE;\bZ)/(p,\surd 0)$.
In this paper, we simply write
\[H^*(BG)=H^*(BG;\bZ)/(p,\surd 0)\]
and study them mainly.  The cohomology $H^{\odd}(BG;\bZ_{(p)})$ 
and the nilpotents
parts in $H^{\even}(BG;\bZ_{(p)})$ are given in \fullref{sec:sec11}.
\fullref{sec:sec4} is devoted to the explanations of stable splitting of $BG$
according to Dietz, Martino and Priddy. In \fullref{sec:sec5}, and 
\fullref{sec:sec6}, we study 
cohomology and stable splitting of $BG$ for a finite group $G$
having a $3$--Sylow  group $(\bZ/3)^2$ or $E=3_+^{1+2}$ respectively.
In \fullref{sec:sec7} and 
\fullref{sec:sec8}, we study cohomology of $BG$ for groups $G$ having 
a $7$--Sylow subgroup $E=7_+^{1+2}$, and the three exotic $7$--local
finite groups. In \fullref{sec:sec9}, we study their  stable splitting.
In \fullref{sec:sec10} we study the cohomology and stable splitting of
the Monster group $\bM$ for $p=13$. 

\section[p--local finite groups over E]{$p$--local finite groups over $E$}
\label{sec:sec2}

Recall that the extraspecial $p$--group $p_+^{1+2}$ has a presentation 
as
\[ p_+^{1+2}={\langle}a,b,c|a^p=b^p=c^p=1, [a,b]=c,\ c\in 
\mathrm{Center}{\rangle}\]
and denote it simply by $E$ in this paper. We consider $p$--local 
finite groups 
over $E$, which are generalization  of groups whose 
$p$--Sylow subgroups are isomorphic to $E$.

The concept of the $p$--local finite groups arose in the work of 
Broto, Levi and Oliver \cite{B-L-O} as a generalization of a classical concept 
of finite groups.  The $p$--local finite group is stated as a triple 
${\langle}S,F,L{\rangle}$ where $S$ is a $p$--group, 
$F$ is a saturated fusion system over a centric linking system $L$ 
over $S$ 
(for a detailed definition, see \cite{B-L-O}). Given a $p$--local finite 
group, we can 
construct its classifying space $B{\langle}S,F,L{\rangle}$ by the 
realization 
$|L|_p^{\wedge}$.  Of course if ${\langle}S,F,L{\rangle}$ is induced 
from a finite 
group $G$ having $S$ as a 
$p$--Sylow subgroup,
then $B{\langle}S,F,L{\rangle}\cong BG$.  However note that in 
general, there exist 
$p$--local finite groups which are not induced from finite groups 
(exotic cases).

Ruiz and Viruel recently determined 
${\langle}p_+^{1+2},F,L{\rangle}$ for 
all odd primes $p$. 
We can check the possibility of existence of finite groups 
only for simple groups and their extensions.  Thus they find 
new exotic $7$--local finite groups.

The $p$--local finite groups ${\langle}E,F,L{\rangle}$ are classified 
by $\Out_F(E)$,
number of $F^{\ec}$--radical $p$--subgroup $A$ (where $A\cong 
(\bZ/p)^2$),
and $\Aut_F(A)$ (for details see \cite{R-V}).  When a $p$--local finite 
group 
is induced from a finite group $G$, then we see easily that 
$\Out_F(E)\cong W_G(E) (=N_G(E)/E.C_G(E))$ and $\Aut_F(A)\cong W_G(A)$. 
Moreover $A$ is $F^{\ec}$--radical 
if and only if $\Aut_F(A)\supset SL_2(\bF_p)$ by \cite[Lemma 4.1]{R-V}.
When $G$ is a sporadic simple group,  $F^{\ec}$--radical follows
$p$--pure.
\begin{thm}[Ruiz and Viruel \cite{R-V}]\label{thm:thm2.1}
If $p\not =3,7,5,13$, then a $p$--local finite group 
${\langle}E,F,L{\rangle}$ is 
isomorphic
to one of the following types.
\begin{enumerate}
\item $E \colon W$ for  $W\subset \Out(E)$ and  $(|W|,p)=1$,
\item $ p^2 \colon SL_2(\bF_p).r$ for  $r|(p-1)$,
\item $SL_3(\bF_p) \colon H$ for $H\cong \bZ/2,\bZ/3$ or $S_3.$
\end{enumerate}
When $p=3,5,7$ or $13$, it is either of one of the previous types or 
of the following types.
\begin{enumerate}\setcounter{enumi}{4}
\item $^2F_4(2)',J_4$, for $p$=3,
\item $Th$ for $p$=5, 
\item $He,He \colon 2,{Fi'}_{24},Fi_{24},O'N,O'N \colon 2$,  
and  three exotic $7$--local finite groups for $p$=7,
\item $\bM$ for $p$=13.
\end{enumerate}
\end{thm}

For case (1), we know that
$H^*(E{\colon}W)\cong H^*(E)^W$.  Except for these extensions and 
exotic cases,
all $H^{\even}(G;\bZ)_{(p)}$ are studied by Tezuka and Yagita \cite{T-Y}. In 
\cite{Y1}, the author studied
ways to distinguish
$H^{\odd}(G;\bZ)_{(p)}$ and $H^*(G;\bZ/p)$ from 
$H^{\even}(G;\bZ)_{(p)}$. 
The stable splittings for such $BG$ are studied in \cite{Y2}.
However there were some errors in the cohomology of 
$^2F_4(2)',{Fi'}_{24},\bM$.
In this paper, we study cohomology and stable splitting of $BG$
for $p=3$,$7$ and $13$ mainly.

\section{Cohomology}
\label{sec:sec3}

In this paper we mainly consider the cohomology 
$H^*(BG;\bZ)/(p,\surd 0)$
where $\surd 0$ is the ideal generated by nilpotent elements.  So we 
write it simply
\[H^*(BG)=H^*(BG;\bZ)/(p,\surd 0).\]
Hence we have
\[H^*(B\bZ/p)\cong \bZ/p[y],\qua H^*(B(\bZ/p)^2)\cong \bZ/p[y_1,y_2]
\text{ with }|y|=|y_i|=2.\]
Let us write $(\bZ/p)^2$ as $A$  
and let an $A$--subgroup of $G$ mean a subgroup
isomorphic to $(\bZ/p)^2$.

The cohomology of the extraspecial $p$ group $E=p_+^{1+2}$ is well 
known.
In particular recall (Leary \cite{L2} and Tezuka--Yagita \cite{T-Y})
\begin{equation}\label{eqn:eqn3.1}
H^*(BE)\cong \left(\bZ/p[y_1,y_2]/(y_1^py_2-y_1y_2^p)\oplus 
\bZ/p\{C\}\right)\otimes \bZ/p[v],
\end{equation}
where $|y_i|=2,|v|=2p,|C|=2p-2$ and $Cy_i=y_i^{p}$, 
$C^2=y_1^{2p-2}+y_2^{2p-2}-y_1^{p-1}y_2^{p-1}$.
In this paper we  write $y_i^{p-1}$ by $Y_i$, and $v^{p-1}$ by $V$,
eg $C^2=Y_1^2+Y_2^2-Y_1Y_2$.
The Poincare series of the subalgebra generated by $y_i$ and $C$ are 
computed
\[\frac{1-t^{p+1}}{(1-t)(1-t)}+t^{p-1}=
\frac{(1+\cdots+t^{p-1})+t^{p-1}}{(1-t)}
=\frac{(1+\cdots+t^{p-1})^2-t^{2p-2}}{(1-t^{p-1})}.\]
From this  Poincare series and \eqref{eqn:eqn3.1}, we get the another
expression of $H^*(BE)$
\begin{equation}\label{eqn:eqn3.2}
H^*(BE)\cong \bZ/p[C,v]\left\{y_1^iy_2^j|0\le i,j\le p-1,
(i,j)\not =(p-1,p-1)\right\}.
\end{equation}
The $E$ conjugacy classes of $A$--subgroups are written by
\begin{align*}
A_i &={\langle}c,ab^i{\rangle} \text{ for }0\le i\le p-1 \\ 
A_{\infty} &={\langle}c,b{\rangle}.
\end{align*}
Letting $H^*(BA_i)\cong \bZ/p[y,u]$ and writing $i_{A_i}^*(x)=x|A_i$
for the inclusion $i_{A_i}{\co}A_i\subset E$, the restriction 
images are given by
\begin{align}\label{eqn:eqn3.3}
y_1|A_i=y \text{ for } i\in \bF_p, y_1|A_{\infty}=0,
\qua & y_2|A_i=iy \text{ for } i\in \bF_p, y_2|A_{\infty}=y,\\   \nonumber
C|A_i=y^{p-1},\qua & v|A_i=u^p-y^{p-1}u\text{ for all }i.
\end{align}
For an element $g=\bigl(\begin{smallmatrix} \alpha&\beta\\ \gamma&\delta
\end{smallmatrix}\bigr)
\in GL_2(\bF_p)$, we can identify  $GL_2(\bF_p)\cong  \Out(E)$ by
\[g(a)=a^{\alpha}b^{\gamma},
g(b)=a^{\beta}b^{\delta},\ g(c)=c^{\det(g)}.\]
Then the action of $g$  on the cohomology is given (see Leary \cite{L2} 
and Tezuka--Yagita \cite[page 491]{T-Y}) by
\begin{equation}\label{eqn:eqn3.4}
g^*C=C,\ g^*y_1=\alpha y_1+\beta y_2,
\ g^*y_2=\gamma y_1+\delta y_2,\ g^*v=(\det(g))v.
\end{equation}
Recall that $A$ is $F^{\ec}$--radical if and only if $SL_2(\bF_p)
\subset W_G(A)$ (see Ruiz--Viruel \cite[Lemma 4.1]{R-V}).

\begin{thm}[Tezuka--Yagita {{\cite[Theorem 4.3]{T-Y}}},
Broto--Levi--Oliver \cite{B-L-O}]
\label{thm:thm3.1}
Let $G$ have the $p$--Sylow subgroup $E$, then we have the isomorphism
\[H^*(BG)\cong H^*(BE)^{W_G(E)}\cap 
_{A{\colon}F^{\ec}-\rad}i_A^{*-1}H^*(BA)
^{W_G(A)}.\]
\end{thm}
In \cite{B-L-O} and \cite{T-Y}, proofs of the above theorem are  given 
only for $H^*(BG;\bZ_{(p)})$. A proof for $H^*(BG)$ is explained in 
\fullref{sec:sec11}.

\section{Stable splitting}
\label{sec:sec4}

Martino--Priddy prove the following theorem of complete stable 
splitting.
\begin{thm}[Martino--Priddy \cite{M-P}]
\label{thm:thm4.1}
Let $G$ be a finite group with a $p$--Sylow subgroup $P$.
The complete stable splitting of $BG$ is given by 
\[BG\sim \vee \rank A(Q,M)X_M\]
where indecomposable summands $X_M$ range over isomorphic classes of 
simple \linebreak $\bF_p[\Out(Q)]$--modules $M$ and over isomorphism classes of 
subgroups 
$Q\subset P$.
\end{thm}

\begin{rem}This theorem also holds for $p$--local finite groups 
over $P$,
because all arguments for the proofs are done about the induced maps
from some fusion systems of $P$ on stable homotopy types of
related classifying spaces.
\end{rem}

For the definition of $\rank A(Q,M)$ see Martino and Priddy \cite{M-P}.  In particular, 
when $Q$ is not a subretract
(that is not a proper retract of a  subgroup) of $P$ 
(see \cite[Definition 2]{M-P}) and when 
$W_G(Q)\subset \Out(Q)\cong GL_n(\bF_p)$
(see \cite[Corollary 4.4 and the proof of Corollary 4.6]{M-P}), 
the rank of $A(Q,M)$ is computed by
\[\rank A(Q,M)=\sum \dim_{\bF_p}(\wbar W_G(Q_i) M),\]
where $\wbar W_G(Q_i)=\sum_{x\in W_G(Q_i)}x$ in $\bF_p[GL_n(F_p)]$ and 
$Q_i$ ranges 
over representatives of $G$--conjugacy classes of 
subgroups isomorphic to $Q$.

Recall that $\Out(E)\cong \Out(A)\cong GL_2(\bF_p)$. The simple
modules of $G=GL_2(\bF_p)$ are well known.  Let us think of 
$A$ as  the natural two-dimensional representation, and $\det$ the 
determinant 
representation of $G$.  Then there are $p(p-1)$ simple 
$\bF_p[G]$--modules given by 
$M_{q,k}=S(A)^q\otimes (\det)^k$ for $0\le q\le p-1,0\le k\le p-2$.  
Harris and Kuhn \cite{H-K} determined the stable splitting
of abelian $p$--groups.  In particular, they showed
\begin{thm}[Harris--Kuhn \cite{H-K}]
\label{thm:thm4.2}
Let $\tilde X_{q,k}=X_{M_{q,k}}$ (resp. $L(1,k)$) 
identifying $M_{q,k}$ as an $\bF_p[\Out(A)]$--module (resp. $M_{0,k}$ as 
an
$\bF_p[\Out(\bZ/p)]$--module).
There is the complete stable splitting 
\[BA\sim \vee_{q,k}(q+1)\tilde X_{q,k}\vee _{q\not =0}(q+1)L(1,q),
\]
where $0\le q\le p-1$, $0\le k\le p-2$. 
\end{thm}
The summand $L(1,p-1)$ is usually written by $L(1,0)$.

It is also known $H^+(L(1,q))\cong  \bZ/p[y^{p-1}]\{y^q\}. $
Since we have the isomorphism
\[H^{2q}(BA)\cong (\bZ/p)^{q+1} \cong H^{2q}((q+1)L(1,q)),
\text{ for }1\le q\le p-1,\]
we get  $H^*(\tilde X_{q,k})\cong 0$ for $*\le 2(p-1)$.
\begin{lemma}\label{lem:lem4.3}
Let $H$ be a finite solvable group with $(p,|H|)=1$ and 
$M$ be an  $\bF_p[H]$--module.
Then we have $\wbar H(M)=(\sum_{x\in H}x)M\cong M^H\cong H^0(H;M)$.
\end{lemma}
\begin{proof}
First assume $H=\bZ/s$ and $x\in \bZ/s$ its generator.  Then
\[\wbar H(M)=(1+x+\cdots+x^{s-1})H.\]
Since $(1-x^s)=0$, we see $\Ker(1-x)\supset \mathrm{Image}(\wbar H)$.
The facts that  $M$ is a $\bZ/p$--module and $(|H|,p)=1$ imply 
$H^*(H;M)=0$
for $*>0$.  Hence
\[\Ker(1-x)/\mathrm{Image}(1+\cdots+x^{s-1})\cong H^1(H;M)=0.\]
Thus we have $\wbar H(M)=\Ker(1-x)=M^H$.

Suppose that $H$ is a group such that 
\[0\to H'\to H\stackrel{\pi}{\to} H''\to 0\]
and  that $\wbar H'(M')=(M')^{H'}$ (resp. $\wbar H''(M'')=(M'')^{H''}$) 
for each
$\bZ/p[H']$--module $M'$ (resp. $\bZ/p[H'']$--module $M''$).
Let $\sigma$ be a (set theoretical) section of $\pi$ and denote
$\sigma(\wbar H'')=\sum_{x\in H''}\sigma(x)\in \bF_p[H]$.
Then 
\[\wbar H(M)=\sigma(\wbar H'') \wbar H'(M)=
\sigma(\wbar H'')(M^{H'})=\wbar H''(M^{H'})=(M^{H'})^{H''}=M^H\]
here the third equation follows from that we can identify
$M^{H'}$ as an $\bF_p[H'']$--module.
Thus the lemma is proved.\end{proof}

It is known from a result of Suzuki \cite[Chapter 3 Theorem 6.17]{S} 
that any subgroup of 
$SL_2(\bF_{p^n})$, whose order is prime to $p$ 
is isomorphic to a subgroup of   
$\bZ/s$, $4S_4$, $SL_2(\bF_3)$, $SL_2(\bF_5)$ or 
\[Q_{4n}={\langle}x,y|x^n=y^2,y^{-1}xy=x^{-1}{\rangle}.\]

\begin{cor}\label{cor:cor4.4}
Let $H\subset GL_2(\bF_p)$ with $(|H|,p)=1$ and  $H$ do not have a 
subgroup
isomorphic to
$SL_2(\bF_3)$ nor $SL_2(\bF_5)$.    Let $G=A{\co}H$ and let us 
write
$BG\sim
\vee _{q,k}\tilde n(H)_{q,k}\tilde X_{q,k}\vee _{q'}\tilde 
m(H)_{q'}L(1,q').$
Then 
\begin{align*}
\tilde n(H)_{q,k} &=\rank_pH^0(H;M_{q,k}), \\
\tilde m(H)_{q'} &=\rank_pH^{2q'}(BG).
\end{align*}
In particular $\tilde n(H)_{q,0}=\rank_pH^{2q}(BG)$.
\end{cor}

\begin{proof}
Since $H^*(\tilde X_{q,k})\cong 0$ for $*\le 2(p-1)$,
it is immediate that $\tilde m(H)_{q'}=\rank_pH^{2q'}(G)$.
Since $ GL_2(\bF_p)\cong SL_2(\bF_p).\bF_p^*$ and $\bF_p^*\cong 
\bZ/(p-1)$, 
each subgroup $H$
in the above satisfies the condition in \fullref{lem:lem4.3}.
The first equation is  immediate from the lemma.
\end{proof}

Next consider the stable splitting for the extraspecial $p$--group $E$.
Dietz and Priddy prove the following theorem.
\begin{thm}[Dietz--Priddy \cite{D-P}]
\label{thm:thm4.5}
Let $X_{q,k}=X_{M_{q,k}}$ (resp. $L(2,k)$, $L(1,k)$) 
identifying $M_{q,k}$ as an $\bF_p[\Out(E)]$--module (resp. $M_{p-1,k}$ 
as an
$\bF_p[\Out(A)]$--module, $\bF_p[\Out(\bZ/p)]$--module).  
There is the complete stable splitting 
\[BE\sim \vee_{q,k}(q+1)X_{q,k}\vee _{k}(p+1)L(2,k)\vee_{q\not 
=0}(q+1)L(1,q)\vee 
L(1,p-1)\]
where $0\le q\le p-1$, $0\le k\le p-2$.
\end{thm}

\begin{rem} Of course $\tilde X_{q,k}$ is different from $X_{q,k}$ 
but
$\tilde X_{p-1,k}=L(2,k)$.
\end{rem}

The number of $L(1,q)$ for $1\le q{<}p-1$ is given by the following.
Let us consider the decomposition $E/{\langle}c{\rangle}\cong \wbar 
A_i\oplus \wbar 
A_{-i}$
where $\wbar A_i={\langle}ab^i{\rangle}$ and $\wbar A_{-0}=\wbar 
A_{\infty}$. 
We consider the projection 
$\pr_i \co E\to  \wbar A_i.$
Let $x\in H^1(B\wbar A_i;\bZ/p)=\Hom(\wbar A_i,\bZ/p)$ be the dual of 
$ab^i$. 
Then
\[\pr_i^*x(a)=x(\pr_{i}(a))=x(\pr_i(ab^iab^{-i})^{1/2})=x((ab^i)^{1/2})=1
/2,\]
\[\pr_i^*x(b)=x(\pr_{i}(ab^i(ab^{-i})^{-1})^{1/(2i)})=1/(2i).\]
Hence for $\beta(x)=y$, we have $\pr_i^*(y)=1/2y_1+1/(2i)y_2$. 
Therefore the $k+1$ elements 
$(1/2y_1+1/(2i)y_2)^k,\
i=0,\ldots,k$ form a base of $H^{2k}(E/{\langle}c{\rangle};\bZ/p)\cong 
(\bZ/p)^{k+1}$
for $k{<}p-1$.
Thus we know the number of $L(1,k)$ is $k+1$ for $0{<}k{<}p-1$.

Recall that 
\[H^{2q}(BE)\cong 
\begin{cases}(\bZ/p)^{q+1}\cong H^{2q}((q+1)L(1,q))\text{ for }0\le 
2\le p-2 \\
(\bZ/p)^{q+2}\cong H^{2p-2}((p+1)L(1,0))\text{ for }q=p-1.
\end{cases}\]
This shows $H^*(X_{q,k})\cong 0$ for $*\le 2p-2$ since so is 
$L(2,k)$. 
The number $n(G)_{q,k}$ of $X_{q,k}$ is only depend on $W_G(E)=H$.  
Hence we have the following corollary.       

\begin{cor}\label{cor:cor4.6}  
Let $G$ have the $p$--Sylow subgroup $E$ and $W_G(E)=H$.  Let
\[ BG \sim \vee  n(G)_{q,k} X_{q,k}  \vee m(G,2)_k L(2,k)\vee
m(G,1)_k  L(1,k). \]
Then $n(G)_{q,k}=\tilde n(H)_{q,k}$ and $m(G,1)_k=\rank_pH^{2k}(G)$.
\end{cor}

Let $W_G(E)=H$. We also compute the dominant summand by the 
cohomology $H^*(BE)^H\cong H^*(B(E{\co}H))$. 
Let us write the $\bZ/p$--module
\[X_{q,k}(H)= S(A)^q\otimes v^k \cap H^*(B(E{\co}H))\quad with\
\ S(A)^q=\bZ/p\{y_1^q,y_1^{q-1}y_2,\ldots,y_2^q\}.\]
Since the module $\bZ/p\{v^k\}$ is isomorphic to the $H$--module  
$\det^k$,
we have the following lemma.
\begin{lemma}  \label{lem:lem4.7} The  number  $n_{q,k}(G)$ of $X_{q,k}$ 
in  $BG$ is given by $\rank_p(X_{q,k}(W_G(E)))$.
\end{lemma}

Next problem is to seek $m(G,2)_k$.
The number $p+1$ for the summand $L(2,k)$ in $BE$ is given as follows.
For each $E$--conjugacy class of $A$--subgroup 
$A_i={\langle}c,ab^i{\rangle},i\in 
\bF_p\cup
\infty$, we see 
\[W_E(A_i)=N_E(A_i)/A_i=E/A_i\cong \bZ/p\{b\}\quad
b^*{\co}ab^i\mapsto ab^ic.\] 
Let $u=\left(\begin{smallmatrix}1&1 \\ 0&1 \end{smallmatrix}\right)$ in
$GL_2(\bF_p)$ and $U={\langle}u{\rangle}$ the maximal unipotent 
subgroup.
Then we can identify $W_E(A_i)\cong U$ by $b\mapsto u$.
For $y_1^sy_2^l \in
M_{q,k}$ (identifying $H^*(BA)\cong S^*(A)=\bZ/p[y_1,y_2]$), we
can compute
\begin{eqnarray*}
\wbar{W}_E(A)y_1^sy_2^l = &(1+u+\cdots+u^{p-1})y_1^sy_2^l
 & = \sum_{i=0}^{p-1}(y_1+iy_2)^sy_2^l  \\
  = &\sum_i\sum_t \tbinom{s}{t}
%\left( \begin{array}{c} s\\ t\end{array} \right)
i^ty_1^{s-t}y_2^ty_2^l & =  \sum_t
\tbinom{s}{t}
%\left( \begin{array}{c}s\\ t\end{array} \right)
\sum_i i^t
y_1^{s-t}y_2^{t+l} .
\end{eqnarray*}
Here $\sum_{i=0}^{p-1}i^t=0$ for $1\le t\le  p-2$, and $=-1$ for 
$t=p-1$.
Hence we know
\[\dim_p \wbar{W}_G(A_i)M_{q,k}=
\begin{cases}0  & \text{ for }1\le q \le p-2 \\
1 & \text{ for }q=p-1.
\end{cases}
\]
Thus we know that $BE$ has just one $L(2,k)$ 
for each $E$--conjugacy $A$--subgroup $A_i$.

\begin{lemma} \label{lem:lem4.8}
Let $A$ be an $F^{\ec}$--radical subgroup, ie $W_G(A)\supset 
SL_2(\bF_p)$.
Then $\wbar W_G(A)(M_{q,k})=0$ for all $k$ and $1\le q \le p-1$.
\end{lemma}
\begin{proof}
The group $SL_2(\bF_p)$ is generated by $u=\left(\begin{smallmatrix}
1&1\\ 0&1 \end{smallmatrix}\right)$ and $u'=\left(\begin{smallmatrix}
1&0\\ 1&1 \end{smallmatrix}\right)$.  We know
$\Ker(1-u)\cong \bZ/p[y_1^p-y_2^{p-1}y_1,y_2]$ and 
$\Ker(1-u')\cong \bZ/p[y_2^p-y_1^{p-1}y_2,y_1]$.
Hence we get
$(\Ker(1-u)\cap \Ker(1-u'))^*\cong 0$ for $0{<}*\le p-1$.
\end{proof}
\begin{prop} \label{prop:prop4.9}
Let $G$ have the $p$--Sylow subgroup $E$.  The number of $L(2,0)$ in 
$BG$ 
is given by
\[m(G,2)_0= \sharp_G(A)-\sharp_G(F^{\ec}A)\]
where $\sharp_G(A)$(resp.$\sharp_G(F^{\ec}A)$) is 
the number of $G$--conjugacy classes
of $A$--subgroups (resp. $F^{\ec}$--radical subgroups).
\end{prop}
\begin{proof}
Let us write $K=E{\co}W_G(E)$ and $H^*(BE)^{W_G(E)}=H^*(BK)$. 
From \fullref{thm:thm3.1}, we have 
\begin{equation} \label{eqn:star}
\quad H^*(BG)\cong H^*(BK)\cap 
_{A{\co}F^{\ec}-\rad}i_A^{*-1}H^*(BA)
^{W_G(A)}.
\end{equation}
Let $A$ be an $A$--subgroup of $K$ and $x\in W_{K}(A)$.
Recall $A={\langle}c,ab^i{\rangle}$ for some $i$.
Identifying $x$ as an element of $N_G(A)\subset E{\co}\Out(E)$
We see $x{\langle}c{\rangle}={\langle}c{\rangle}$ from \eqref{eqn:eqn3.4} 
and since 
${\langle}c{\rangle}$ is the center of $E$. 
Hence
\[W_{K}(A)\subset B=U{\co}(\bF_p^*)^2 \text{ the Borel subgroup}.\]
So we easily see that $\wbar W_{K}(y_1^{p-1})=\lambda y_2^{p-1}$
for some $\lambda\not =0$ follows from $b^*y_i^{p-1}=y_i^{p-1}$ for
$b= \text{diagonal} \in (\bF_p)^{*2}$ and the arguments just before
\fullref{lem:lem4.8}.
We also see $\wbar W_K(y_1^{p-1-i}y_2^i)=0$
for $i{>}0$. 
Hence we have $m(K,2)_0=\sharp _K(A)$.
From the isomorphism \eqref{eqn:star}, we have 
$m(G,2)_0=\sharp_K(A)-\sharp_G(F^{\ec}A)$.  

On the other hand $m(G,2)_0\le \sharp_G(A)-\sharp_G(F^{\ec}A)$ 
from the above lemma.
Since $\sharp_K(A)\ge \sharp_G(A)$, we see that 
$\sharp_K(A)=\sharp_G(A)$ 
and get the proposition.
\end{proof}
\begin{cor} \label{cor:cor4.10}
Let $G$ have the $p$--Sylow subgroup $E$.  The number of $L(1,0)$ in 
$BG$ 
is given by
\[m(G,1)_{p-1}=\rank_pH^{2(p-1)}(G)= \sharp_G(A)-\sharp_G(F^{\ec}A).\]
\end{cor}

\begin{proof}
Since $L(1,0)=L(1,p-1)$ is linked to $L(2,0)$, we know
$m(G,1)_{p-1}=m(G,2)_0$.
\end{proof}

\begin{lemma} \label{lem:lem4.11}
Let $\xi\in \bF_p^*$ be a primitive $(p-1)$th root of $1$
and  $G\supset E{\colon}{\langle}\diag(\xi,\xi){\rangle}$. If 
$\xi^{3k}\not =1$, 
then $BG$ does not contain the summand $L(2,k)$, ie $m(G,2)_k=0$.
\end{lemma}

\begin{proof}
It is sufficient to prove the case 
$G=E{\colon}{\langle}\diag(\xi,\xi){\rangle}$.
Let $G=E{\colon}{\langle}\diag(\xi,\xi){\rangle}$.  Recall 
$A_i={\langle}c,ab^i{\rangle}$ and
\[ \diag(\xi,\xi)\co ab^i \mapsto (ab^i)^{\xi},\qua c\mapsto 
c^{\xi^2}.\]
So the Weyl group is 
$W_G(A_i)=U{\co}{\langle}\diag(\xi^2,\xi){\rangle}$.
For  $v=\lambda y_1^{p-1}+\cdots \in M_{q,k}$, we have
\[\wbar W_G(A_i)v=\sum_{i=0}^{p-2}(\xi^{3i})^k 
\diag(\xi^{2i},\xi^{i})(1+\cdots+u^{p-1})v=\sum_{i=0}^{p-2}\xi^{3ik}
\lambda y_2^{p-1}.\]
Thus we get the lemma from 
$\sum_{i=0}^{p-2}\xi^{3ik}=0$ for $3k\not =0 \mod(p-1)$ and $=-1$ 
otherwise.
\end{proof}

\section[Cohomology and splitting of B(Z/3)2]
{Cohomology and splitting of $B(\bZ/3)^2$}
\label{sec:sec5}

In this section, we study the cohomology and stable splitting
of $BG$ for $G$ having a $3$--Sylow subgroup $(\bZ/3)^2=A$. 
In this and next sections, $p$ always means $3$.
Recall $\Out(A)\cong GL_2(\bF_3)$ and $\Out(A)'$ consists the
semidihedral group
\[SD_{16}={\langle}x,y|x^8=y^2=1,yxy^{-1}=x^3{\rangle}.\]
Every $3$--local finite group $G$ over $A$ is of type 
$A{\co}W,\ W\subset SD_{16}$.  There is the $SD_{16}$--conjugacy 
classes of 
subgroups(here $B\longleftarrow C$ means $B\supset C$)
\[SD_{16} \begin{cases}
\longleftarrow Q_8\longleftarrow \bZ/4\\
\longleftarrow \bZ/8 \longleftarrow \bZ/4 \longleftarrow 
\bZ/2 \longleftarrow 0\\
\longleftarrow D_8 \longleftarrow \bZ/2\oplus \bZ/2 \longleftarrow 
\bZ/2
\end{cases}
\]
We can take  generators  of subgroups in $GL_2(\bF_3)$
by the matrices
\begin{align*}
\bZ/8={\langle}l{\rangle}, Q_8={\langle}w,k{\rangle}, 
D_8={\langle}w',k{\rangle}, \bZ/4={\langle}w{\rangle}, \\
\bZ/4={\langle}k{\rangle}, \bZ/2\oplus 
\bZ/2={\langle}w',m{\rangle},
\bZ/2={\langle}m{\rangle}, \bZ/2={\langle}w'{\rangle},
\end{align*}
where $l=\left(\begin{smallmatrix} 0&1\\ 1&-1\end{smallmatrix}\right)$, 
$w=\left(\begin{smallmatrix} 0&1\\ -1&0 \end{smallmatrix}\right)$,
$k=l^2=\left(\begin{smallmatrix} 1&-1\\ -1&-1\end{smallmatrix}\right)$, 
$w'=wl=\left(\begin{smallmatrix} 1&-1\\ 0&-1\end{smallmatrix}\right)$ and
$m=w^2=k^2=\left(\begin{smallmatrix} -1&0\\ 0&-1\end{smallmatrix}\right)$.
Here we note that $k$ and $w$ are $GL_2(\bF_3)$--conjugate,
in fact $uku^{-1}=w$. Hence we note that
\[H^*(B(A{\co}{\langle}k{\rangle}))\cong 
H^*(B(A{\co}{\langle}w{\rangle})).\]
The cohomology of $A$ is given $H^*(BA)\cong \bZ/3[y_1,y_2]$,
and the following
are immediately
\[H^*(BA)^{{\langle}m{\rangle}}\cong \bZ/3[y_1^2,y_2^2]\{1,y_1y_2\} 
\qua 
H^*(BA)^{{\langle}w'{\rangle}}
\cong \bZ/3[y_1+y_2,y_2^2].
\]
Let us write $Y_i=y_i^2$ and $t=y_1y_2$.   The $k$--action is given
$Y_1\mapsto Y_1+Y_2+t$, $Y_2\mapsto Y_1+Y_2-t,$ $t\mapsto -Y_1+Y_2.$
So the following are invariant
\[a=-Y_1+Y_2+t,\ a_1=Y_1(Y_1+Y_2+t),\ a_2=Y_2(Y_1+Y_2-t),\ 
b=t(Y_1-Y_2).\]
Here we note that $a^2=a_1+a_2$ and $b^2=a_1a_2.$ 
We can prove the invariant ring is
\[H^*(BA)^{{\langle}k{\rangle}}\cong \bZ/3[a_1,a_2]\{1,a,b,ab\}.\]
Next consider the invariant under $Q_8={\langle}w,k{\rangle}$.  The 
action for $w$ is
$a\mapsto -a,\ a_1\leftrightarrow a_2,\ b\mapsto b$.  Hence we get
\[H^*(BA)^{Q_8}\cong \bZ/3[a_1+a_2,a_1a_2]\{1,b\}\{1,(a_1-a_2)a\}.\]
Let us write $S=\bZ/3[a_1+a_2,a_1a_2]$ and $a'=(a_1-a_2)a$.
The action for $l$ is given
$l \co Y_1\mapsto Y_2\mapsto Y_1+Y_2+t \mapsto Y_1+Y_2-t \mapsto 
Y_1.$
Hence $l \co a\mapsto -a,\ a_1\leftrightarrow a_2,\ b\mapsto -b$.  
Therefore
we get 
$H^*(BA)^{{\langle}l{\rangle}}\cong S\{1,a',ab,(a_1-a_2)b\}.$

The action for $w' \co Y_1\mapsto Y_1+Y_2+t,\ Y_2\mapsto Y_2$, 
implies that
$w' \co  a\mapsto a,\ a_i\mapsto a_i,\ b\mapsto -b$.  Then we can 
see
\[H^*(BA)^{D_8}=H^*(BA)^{{\langle}k,w'{\rangle}}\cong 
\bZ/3[a_1,a_2]\{1,a\}
\cong S\{1,a,a_1,a'\}.\]
We also have
\[H^*(BA)^{SD_{16}}\cong H^*(BA)^{Q_8}\cap H^*(BA)^{\bZ/8}
\cong S\{1,a'\}.\]
Recall the Dickson algebra $DA=\bZ/3[\tilde D_1,\tilde D_2]\cong 
H^*(BA)^{GL_2(\bF_3)}$ where 
$\tilde D_1=Y_1^3+Y_1^2Y_2+Y_1Y_2^2+Y_2^3=(a_2-a_1)a=a'$ and
$\tilde D_2=(y_1^3y_2-y_1y_2^3)^2=a_1a_2$.
Using $a^2=(a_1+a_2)$ and $\tilde D_1^2=a^6-a_1a_2a^2$, we can write
\[H^*(BA)^{SD_{16}}\cong \bZ/3[a^2,\tilde D_2]\{1,\tilde D_1\}
\cong DA\{1,a^2,a^4\}.\]   

\begin{thm} \label{thm:thm5.1} Let $G=(\bZ/3)^2{\co}H$ for $H\subset SD_{16}$.
Then $BG$ has the stable  splitting given by 
\begin{footnotesize}
\[\stackrel{\tilde X_{0,0}}{\gets} SD_{16} \begin{cases}
\stackrel{\tilde X_{0,1}}{\longleftarrow}
Q_8\\
\quad    \\
\stackrel{\tilde X_{2,1}}{\longleftarrow} \bZ/8 
\stackrel{\tilde X_{2,0}\vee \tilde X_{0,1}\vee L(1,0)}
{\longleftarrow}
\bZ/4 \stackrel{2\tilde X_{2,0}\vee 2\tilde X_{2,1}
\vee 2L(1,0)}{\longleftarrow}
\bZ/2 \stackrel{2\tilde X_{1,0}\vee 2\tilde X_{1,1}
\vee 2L(1,1)}{\longleftarrow} 0\\             \quad    \\
\stackrel{\tilde X_{2,0}\vee L(1,0)}{\longleftarrow} D_8 
\stackrel{\tilde X_{2,0}\vee \tilde X_{2,1}\vee L(1,0)}{
\longleftarrow}
\bZ/2\oplus \bZ/2 \stackrel{\tilde X_{1,0}
\vee \tilde X_{1,1}\vee L(1,1)}{\longleftarrow} 
\bZ/2
\end{cases}
\]
\end{footnotesize}
where 
$\stackrel{\tilde X_1}{\gets} \cdots \stackrel{\tilde X_s}{\gets}H$ 
means
$B((\bZ/3)^2{\colon}H)\sim \tilde X_1\vee \cdots \vee \tilde X_s$.
\end{thm} 
For example
\[B(E{\co}SD_{16})\sim \tilde X_{0,0}, \qua  B(E{\co}Q_8)\sim 
\tilde X_{0,0}\vee 
\tilde X_{0,1}, \qua B(E{\co}\bZ/8)\sim \tilde X_{0,0}\vee \tilde 
X_{2,1}.\]
Main parts of the above splittings are given by the author in
\cite[(6)]{Y2} by direct computations of $\wbar W_G(A)$ (see
\cite[page 149]{Y2}).  However we get the theorem more easily
by using cohomology here.  For example, let us consider the case
$G=A{\co}{\langle}k{\rangle}$.  The cohomology
\[H^0(BG)\cong \bZ/3, \qua H^2(BG)\cong 0,H^4(BG)\cong \bZ/3\]
implies that $BG$ contains just one $\tilde X_{0,0},\tilde 
X_{2,0},L(1,0)$ but
does not $\tilde X_{1,0},L(1,1)$. Since $\det(k)=1$, we also know that
$\tilde X_{0,1},\tilde X_{2,1}$ are contained.  So we can see
\[B(A{\colon}\bZ/4)\sim \tilde X_{0,0}\vee \tilde X_{0,1}\vee \tilde 
X_{2,0}
\vee \tilde X_{2,1} \vee L(1,0).\]
Next consider the case $G'=A \colon {\langle}l{\rangle}$. The fact 
$H^4(G)\cong 0$ 
implies that $BG'$ does not contain $\tilde X_{2,0},L(1,0)$.  The 
determinant $\det(l)=-1$,
and $l \co a\mapsto -a$ shows that $BG'$ contains $\tilde X_{2,1}$ 
but does not contain
$\tilde X_{0,1}$.  Hence we know $BG'\sim \tilde X_{0,0} \vee \tilde 
X_{2,1}$.
Moreover we know $BA \colon SD_{16}\sim \tilde X_{0,0}$ since  
$w \co a\to -a$ but 
$\det(w)=1$.  Thus we have the graph
\[ \stackrel{\tilde X_{0,0}}{\gets} SD_{16} 
\stackrel{\tilde X_{2,1}}{\longleftarrow} \bZ/8 
\stackrel{\tilde X_{2,0}\vee \tilde X_{0,1}\vee L(1,0)}
{\longleftarrow}
\bZ/4 .\] 
Similarly we get the other parts of the above graph.  

\begin{cor} \label{cor:cor5.2}
Let $S=\bZ/3[a_1+a_2,a_1a_2]$. Then we have the isomorphisms
\begin{align*}
H^*(\tilde X_{0,0})& \cong S\{1,\tilde D_1\} \\
H^*(\tilde X_{0,1})& \cong S\{b,\tilde D_1b\} \\   
H^*(\tilde X_{2,1})& \cong S\{ab,(a_1-a_2)b\} \\  
H^*(\tilde X_{2,0}\vee L(1,0))& \cong S\{a,a_1-a_2\}\cong 
DA\{a,a^2,a^3\}.
\end{align*}
\end{cor}
Here we write down the decomposition of cohomology for a typical case
\begin{align*}
H^*(BA)^{{\langle}k{\rangle}} & \cong S\{1,a_1-a_2\}\{1,a\}\{1,b\} \\
 & \cong S\{1,a(a_1-a_2), b,ba(a_1-a_2), ab,(a_1-a_2)b, a,(a_1-a_2)\} \\
 & \cong H^*(\tilde X_{0,0})\oplus H^*(\tilde X_{0,1})
\oplus H^*(\tilde X_{2,1})\oplus H^*(\tilde X_{2,0}\vee L(1,0)).
\end{align*}

\section[Cohomology and splitting of B3(1+2)]
{Cohomology and splitting of $B3_+^{1+2}$.}
\label{sec:sec6}

In this section we study the cohomology and stable splitting of $BG$ 
for
$G$ having a $3$--Sylow subgroup $E=3_+^{1+2}$.
In the splitting for $BE$, the summands $X_{q,k}$ are called 
dominant summands.  Moreover the summands $L(2,0)\vee L(1,0)$ is
usually written by $M(2)$.
\begin{lemma} \label{lem:lem6.1}
If $G\supset E{\co}{\langle}\diag(-1,-1){\rangle}$ identifying 
$\Out(E)\cong 
GL_2(\bF_3)$ and
$G$ has $E$ as a $3$--Sylow subgroup, then
\[BG\sim (\text{dominant summands})\vee (\sharp_G(A)-\sharp_G(F^{\ec}A)
(M(2)).\]
\end{lemma}
\begin{proof}
From \fullref{lem:lem4.11}, 
we know $m(G,2)_1=0$ ie $L(2,1)$ is not 
contained.
The summand $L(1,1)$ is also not contained, since 
$H^2(BE)^{{\langle}\diag(-1,-1){\rangle}}\cong 0$.
The lemma is almost immediately from \fullref{prop:prop4.9}
and \fullref{cor:cor4.10}.
\end{proof}      
\begin{thm} \label{thm:thm6.2}
If $G$ has a $3$--Sylow subgroup $E$, then $BG$ is homotopic to the 
classifying space 
of one of the following groups.
Moreover the stable splitting is given by the graph so that
$\stackrel{X_1}{\gets}\cdots\stackrel{X_s}{\gets}G$ 
means $BG\sim X_1\vee\cdots\vee X_i$ and $EH=E \colon H$ for $H\subset 
SD_{16}$
\begin{footnotesize}
\[\stackrel{X_{0,0}}{\gets}
J_4 \begin{cases}\stackrel{M(2)}{\gets}
ESD_{16} \begin{cases}
\stackrel{X_{0,1}}{\gets}
EQ_8\\
\quad \\
\stackrel{X_{2,1}}{\gets} E\bZ/8 
\stackrel{\stackrel{X_{2,0}\vee X_{0,1}}\vee M(2)} {\gets}
E\bZ/4 \stackrel{\stackrel{2X_{2,0}\vee 2X_{2,1}}{\vee 
2M(2)}}{\gets}
E\bZ/2 \stackrel{\stackrel{2X_{1,0}\vee 2X_{1,1}\vee}{ 
4L(2,1)\vee 2L(1,1)}}{\gets}
E\\
\stackrel{X_{2,0}\vee M(2)}{\gets}  ED_8 
\ \stackrel{X_{2,0}\vee X_{2,1}\vee M(2)}{
\gets}
\ E(\bZ/2)^2 
\stackrel{\stackrel{X_{1,0}\vee X_{1,1}\vee}{2L(2,1)\vee L(1,1)}}
{\gets} 
E\bZ/2
\end{cases}\\
\quad \\
\stackrel{X_{2,0}}{\gets}\  ^2F_4(2)' 
\stackrel{M(2)}{\gets} M_{24}
\stackrel{X_{2,0}\vee X_{2,1}}{\gets} M_{12}
\stackrel{M(2)}{\gets} \bF_3^2{\co}GL_2(\bF_3)
\stackrel{\stackrel{X_{1,0}\vee X_{1,1}\vee}{L(2,1)\vee L(1,1)}}
{\gets} \bF_3^2{\co}SL_2(\bF_3)
\end{cases}\]
\end{footnotesize}
\end{thm}
\begin{proof}
All groups except for $E$,$E{\co}{\langle}w'{\rangle}$ and 
$\bF_3^2{\co}SL_2(\bF_3)$
contain
$E{\co}{\langle}\diag(-1,-1){\rangle}$.  Hence we get the theorem 
from \fullref{cor:cor4.4}, \fullref{thm:thm5.1} and \fullref{lem:lem6.1}, 
except for the place for 
$H^*(BE{\co}{\langle}w'{\rangle})$
and $H^*(\bF_3^2{\co}SL_2(\bF_3))$.

Let $G=E{\co}{\langle}w'{\rangle}$. Note $w'{\co}y_1 \mapsto 
y_1-y_2,
y_2\mapsto -y_2,
v\mapsto -v$.  Hence $H^2(G)\cong \bZ/3\{y_1+y_2\}$. So $BG$ contains 
one 
$L(1,1)$.  Next consider the number of $L(2,0)$, $L(2,1)$.
The $G$--conjugacy classes of $A$--subgroups are $A_0,A_2,A_1\sim 
A_{\infty}$.  The Weyl groups are
\[W_G(A_{\infty})\cong U,\qua W_G(A_2)\cong 
U{\colon}{\langle}\diag(-1,-1){\rangle},
\qua W_G(A_{0})\cong U{\colon}{\langle}\diag(-1,1){\rangle},\]
eg $N_G(A_0)/A_0$ is generated by 
$b,w'$ which is represented by
$u,\diag(-1,1)$ respectively.
By the arguments similar to the proof of \fullref{lem:lem4.11}, 
we have that
\[\begin{cases} \dim(\wbar W_G(A_i)M_{2,0})=1\text{ for all }i\\
\dim(\wbar W_G(A_i)M_{2,1})=1,1,0\text{ for }i=\infty,2,0
\text{ respectively}.
\end{cases}\]
Thus we show $BG\supset 3L(2,0)\vee 2L(2,1)$
and we get the graph for $G=E{\colon}{\langle}w'{\rangle}$.

For the place $G=\bF_3^2{\co}SL_2(\bF_3)$, we see 
$W_G(A_{\infty})\cong SL_2(\bF_3)$.  We also have 
\[\begin{cases} \dim(\wbar W_G(A_i)M_{2,0})=0,1,1 
\text{ for }i=\infty,2,0
\text{ respectively} \\
\dim(\wbar W_G(A_i)M_{2,1})=0,1,0\text{ for } i=\infty,2,0
\text{ respectively}.
\end{cases}\]  
Thus we can see the graph for the place 
$H^*(\bF_3^2{\co}SL_2(\bF_3))$.   
\end{proof}
\begin{rem} From Tezuka--Yagita \cite{T-Y}, Yagita \cite{Y1} and
\fullref{thm:thm2.1}, we have the following 
homotopy equivalences (localized at $3$).
\[BJ_4\cong BRu,\quad BM_{24}\cong BHe, 
\qua BM_{12}\cong BGL_3(\bF_3)\] 
\[ B(E{\co}SD_{16})\cong BG_2(2)\cong BG_2(4),
\qua B(E{\co}D_8)\cong BHJ\cong BU_3(3).\]
\end{rem}

We write down the cohomologies explicitly (see also Tezuka--Yagita
\cite{T-Y} and Yagita \cite{Y2}).  First we compute $H^*(B(E{\co}H))$.
The following cohomologies are easily computed
\begin{align*}
H^*(BE)^{{\langle}m{\rangle}} &\cong 
\bZ/3[C,v]\{1,y_1y_2,Y_1,Y_2\},\qua
H^*(BE)^{{\langle}w{\rangle}}\cong \bZ/3[C,v]\{1,Y_1+Y_2\}.\\
H^*(BE)^{{\langle}k{\rangle}} & \cong \bZ/3[C,v]\{1,a\}
\text{ where }a=-Y_1+Y_2+y_1y_2,\ C^2=a^2.
\end{align*}
Recall that $V=v^{p-1}$ and $C$ multiplicatively generate 
$H^*(BE)^{\Out(E)}$.
Let us write 
\[CA=\bZ/p[C,V]\cong H^*(BE)^{\Out(E)}.\]
Then we have
\begin{align*}
H^*(BE)^{{\langle}w'{\rangle}} & \cong 
CA\{1,y_1',Y_1',Y_2,Y_2y_1',y_2v,y_1'y_2v,Y_1'y_2v\}
\text{ with }y_1'=y_1+y_2 \\
H^*(BE)^{{\langle}w',m{\rangle}} & \cong CA\{1,a,a',Y_2\}\text{ where } 
a'=(t+Y_2)v=y_1'y_2v.
\end{align*}
We can compute 
\begin{align*}
H^*(BE)^{Q_8} & \cong H^*(BE)^{{\langle}k{\rangle}}\cap 
H^*(BE)^{{\langle}w{\rangle}}
\cong \bZ/3[C,v]\cong CA\{1,v\},\\
H^*(BE)^{D_8}& \cong CA\{1,a\},\qua
H^*(BE)^{{\langle}l{\rangle}}\cong CA\{1,av\}.
\end{align*}
Hence we have $H^*(BE)^{SD_{16}}\cong CA.$

Let $D_1=C^p+V$ and $D_2=CV$.  Then it is known that
\[D_1|A_i=\tilde D_1, \ D_2|A_i=\tilde D_2\text{ for all }
i\in \bF_p\cup 
\infty.\]
So we also write $DA\cong\bZ/p[D_1,D_2]$.   Since $CD_1-D_2=C^{p+1}$,
we can write
$CA\cong DA\{1,C,C^2,\ldots,C^p\}.$

Now return to the case  $p=3$ and we get (see \cite{T-Y})
\[H^*(BJ_4)\cong H^*(BE)^{SD_{16}}\cap 
i_0^{*-1}H^*(BA_0)^{GL_2(\bF_3)}
\cong DA.\]
\begin{prop} \label{prop:prop6.3} 
There are isomorphisms for $|a''|=4$,
\[H^*(^2F_4(2)')\cong DA\{1,(D_1-C^3)a''\},\qua 
H^*(M_{24})\cong DA\oplus CA\{a''\}.\]
\end{prop}
\begin{proof}
Let $G=M_{24}$.  Then $G$ has just two $G$--conjugacy classes of 
$A$--subgroups
\[\{A_0, A_2\}, \qua  \{A_1, A_{\infty}\}.\]
It is known that one is $F^{\ec}$--radical and  the other is not.
Suppose that $A_0$ is $F^{\ec}$--radical.  Then $W_G(A_0)\cong 
GL_2(\bF_3)$.
Let $a''=a+C$.  Then
\[a''|A_0=(-Y_1+Y_2+y_1y_2+C)|A_0=0,\qua a''|A_{\infty}=-Y.\]
By \fullref{thm:thm3.1} 
\[H^*(BM_{24})\cong H^*(BE)^{D_8}\cap 
i_{A_{0}}^{*-1}H^*(BA_0)^{W_G(A_0)},\]
we get the isomorphism for $M_{24}$.
When $A_{\infty}$ is a $F^{\ec}$--radical, we take $a''=a-c$. Then we 
get the same result. 

For $G=^2F_4(2)'$, the both conjugacy classes are $F^{\ec}$--subgroups
and $W_G(A_{\infty})\cong GL_2(\bF_3)$.  Hence (for case $a''=a+C$)
\[H^*(B^2F_4(2)')\cong H^*(BM_{24})
\cap i_{A_{\infty}}^{*-1}H^*(BA_{\infty})^{GL_2(\bF_3)}.\]
We know 
\[(D_1-C^3)a''|A_0=0, \qua (D_1-C^3)a''|A_{\infty}=-VY=-\tilde D_2. 
\]
Thus we get the cohomology of $^2F_4(2)'$.
\end{proof}

\begin{rem} In \cite{T-Y,Y2}, we take
\[(\bZ/2)^2={\langle}\diag(\pm 1,\pm 1){\rangle},\quad 
D_8={\langle}\diag(\pm 1,\pm 1),
w{\rangle}. \] 
For this case, the $M_{24}$--conjugacy classes
of $A$--subgroups are $A_0\sim A_{\infty},\ A_1\sim A_2$ , and we can 
take
$a''=C-Y_1-Y_2$.  The expressions of 
$H^*(M_{12})$,
$H^*(A{\colon}GL_2(\bF_3))$ 
become  more simple (see \cite{T-Y,Y2}), in 
fact,
\[H^*(B^2F_4(2)')\cong DA\{1,(Y_1+Y_2)V\}.\]
\end{rem}

\begin{rem} \cite[Corollary 6.3]{T-Y}
and \cite[Corollary 3.7]{Y2} were not correct.  This followed from
an error in \cite[Theorem 6.1]{T-Y}. This theorem is only correct with  
adding the assumption that
there are exactly two $G$ conjugacy classes of $A$--subgroups such that
one is $p$--pure and the other is not. This assumption is always 
satisfied for
sporadic simple groups but not for $ ^2F_4(2)'$.
\end{rem}

\begin{cor} \label{cor:cor6.4}
There are isomorphisms of cohomologies
\begin{align*}
H^*(X_{2,0})& \cong DA\{D_2\},\qua 
H^*(X_{2,1})\cong CA\{av\} 
\text{ where }
(av)^2=CD_2\\
H^*(X_{0,1})& \cong CA\{v\}, \qua 
H^*(M(2))\cong DA\{C,C^2,C^3\}\text{ where }C^4=CD_1-D_2.
\end{align*}
\end{cor}
Here we write down typical examples.  First recall
\begin{align*}
CA &\cong DA\{1,C,C^2,C^3\}\cong H^*(X_{0,0})\oplus H^*(M(2)) \\
CA\{C\} & \cong DA\{C,C^2,C^3,\ D_2\}\cong H^*(M(2))\oplus 
H^*(X_{2,0}).
\end{align*}
Thus the decomposition for $H^*(BE)^{D_8}$  gives the isomorphisms
\[ CA\{1,a''\}\cong CA\{1,C\}\cong 
H^*(X_{0,0})\oplus H^*(M(2))\oplus H^*(X_{2,0})\oplus H^*(M(2)).\]
Similarly the decomposition for $H^*(BE)^{{\langle}k{\rangle}}$ gives 
the isomorphism
\[ CA\{1,a, v, av\}\cong H^*(BE)^{D_8}\oplus H^*(X_{0,1})
\oplus H^*(X_{2,1}).\]
We recall here \fullref{lem:lem4.7} and the module
\[X_{q,k}({\langle}k{\rangle})=S(V)^q\otimes v^k \cap 
H^*(B(E{\co}{\langle}k{\rangle}).\]
Then it is easily seen that
\[X_{0,0}({\langle}k{\rangle})=\{1\},  
X_{2,0}({\langle}k{\rangle})=\{a\},  X_{0,1}
({\langle}k{\rangle})=\{v\},
X_{2,1}({\langle}k{\rangle})=\{av\}. \]
Hence we also see $B(E{\co}{\langle}k{\rangle})$ has the dominant 
summands 
$X_{0,0}\vee X_{2,0}\vee X_{0,1}\vee X_{2,1}.$
Moreover it has non dominant summands $2M(2)$ since 
$H^{4}(B(E{\co}{\langle}k{\rangle}))
\cong \bZ/3\{C,a\} $. 
Thus we can give an another proof of \fullref{thm:thm6.2}
from \fullref{lem:lem4.7} and the cohomologies $H^*(BG)$.

\section[Cohomology for B7(1+2) I]
{Cohomology for $B7_+^{1+2} $ I.}
\label{sec:sec7}

In this section, we assume $p=7$ and $E=7_+^{1+2}$.
We are interested in groups 
$O'N,O'N \colon 2,He,He{\co}2,Fi_{24}',Fi_{24}$
and three exotic $7$--local groups. Denote them by $RV_1,RV_2,RV_3$
according the numbering in \cite{R-V}.
We have the diagram from Ruiz and Viruel
\begin{footnotesize}
\[\begin{cases}\stackrel{3SD_{32}}{\longleftarrow}
\    \stackrel{SL_2(\bF_7){\colon}2}{RV_3}
\    \stackrel{3SD_{16}}{\longleftarrow}
\    \stackrel{SL_2(\bF_7){\colon}2,SL_2(\bF_7){\colon}2}{RV_2}
\    \stackrel{3SD_{16}}{\longleftarrow}
\    \stackrel{SL_2(\bF_7){\colon}2}{O'N{\colon}2}
\    \stackrel{3D_{8}}{\longleftarrow}
\    \stackrel{SL_2(\bF_7){\colon}2,Sl_2(\bF_7){\colon}2}{O'N}        
        \\ 
\quad \\
\stackrel{6^2{\colon}2}{\longleftarrow}
\stackrel{SL_2(\bF_7){\colon}2, GL_2(\bF_7)}{RV_1}
\stackrel{6^2{\colon}2}{\longleftarrow}
\stackrel{SL_2(\bF_7){\colon}2}{Fi_{24}}
\stackrel{6S_3}{\longleftarrow}
\stackrel{SL_2(\bF_7){\colon}2,SL_2(\bF_7){\colon}2}{Fi_{24}'}
\stackrel{6S_3}{\longleftarrow}
\stackrel{SL_2(\bF_7){\colon}2}{He{\colon}2} 
\stackrel{3S_3}{\longleftarrow}
\stackrel{SL_2(\bF_7)}{He} 
\end{cases} \]
\end{footnotesize}
Here $  \stackrel{H}{\longleftarrow} 
\stackrel{W_1,\ldots,W_2}{G}$
means $W_G(E)\cong H, W_i=W_G(A_i)$ for $G$--conjugacy classes of 
$F^{\ec}A-$
subgroups $A_i$.

In this section, we study the cohomology of $O'N,RV_2,RV_3$.
First we study the cohomology of $G=O'N$.  
The multiplicative generators of
$H^*(BE)^{3D_8}$ are still studied in \cite[Lemma 7.10]{T-Y}.
We will study more detailed cohomology structures here.

\begin{lemma} \label{lem:lem7.1} There is the $CA$--module isomorphism
$$
H^*(BE)^{3D_8} \cong CA\{1,a,a^2,a^3/V,a^4/V,a^5/V,
b,ab/V,a^2b/V,  d,ad,a^2d\}, 
$$
where $a=(y_1^2+y_2^2)v^2$,$b=y_1^2y_2^2v^4$ and  
$d=(y_1y_2^3-y_1^3y_2)v$.
\end{lemma}

\begin{proof}
The group  $3D_8\subset GL_2(\bF_7)$
is generated by $\diag(-1,1),(2,2)$ and $w=\left(\begin{smallmatrix} 
0&1\\ -1&0\end{smallmatrix}\right)$.  
If $y_1^iy_2^jv^k$ is invariant under $\diag(-1,1),\diag(1,-1)$
and $\diag(2,2)$, then $i=j=k \mmod(2)$ and $i+j+2k=0 \mmod(3)$.
When $i,j\le 6,k\le 5$ but $(i,j)\not =(6,6)$, 
the invariant monomials have the following terms,
$y_1^2v^2$, $y_1^4v^4$, $y_1^6$,
$y_1^2y_2^2v^4$, $y_1^4y_2^4v^2$, $y_1y_2v^5$, $y_1^3y_2^3v^3$,
$y_1^5y_2^5v$, $y_1^2y_2^4$, $y_1^2y_2^6v^2$, $y_1^4y_2^6v^4$,
$y_1y_2^3v$, $y_1y_2^5v^3$, $y_1^3y_2^5v^5$
and terms obtained by exchanging $y_1$ and $y_2$.  Recall that
$w \co y_1\mapsto y_2,y_2\mapsto -y_1$ and $v\to v$. 
From the expression of \eqref{eqn:eqn3.2}, we have 
\[H^*(BE)^{3D_8}\cong CA\{1,a,a^2,a', b,b',
c,c',c'', d,ad,bd\}\]
where $a=(y_1^2+y_2^2)v^2$,$a'=y_1^6+y_2^6$,$b=y_1^2y_2^2v^4$,
$b'=y_1^4y_2^4v^2$, $c=(y_1^2y_2^4+y_1^4y_2^2)$,
$c'=(y_1^2y_2^6+y_1^6y_2^2)v^2$ , $c''= (y_1^4y_2^6+y_1^6y_2^4)v^4$,
$d=(y_1y_2^3-y_1^3y_2)v$, $ad=(y_1y_2^5-y_1^5y_2)v^3$
and $bd=(y_1^3y_2^5-y_1^5y_2^3)v^5$.
Here $a^2d=bd$ from $(y_1^6-y_2^6)y_1y_2=0$ in $H^*(BE)$.
It is easily seen that 
$b'V=b^2$, $cV=ab$, $c'V=(a^2-2b)b$ and $c''V=ab^2$.
Moreover we get 
\begin{align*}
a^3/V& =(y_1^2+y_2^2)^3 =
(y_1^6+y_2^6)+3y_1^2y_2^2(y_1^2+y_2^2)=a'+3ab \\
a^4/V& =(y_1^2+y_2^2)^4v^2=
((y_1^8+y_2^8)+4y_1^2y_2^2(y_1^4+y_2^4)+6y_1y_2^4)v^2 \\
& =aC+4c'+6b' \\
a^5/V & =((y_1^{10}+y_2^{10})+5y_1^2y_2^2(y_1^6+y_2^6)+
10y_1^4y_2^4(y_1^2+y_2^2))v^4 \\
& =c'C+10bC+10c''.
\end{align*}
Hence, we can take generators $a^4/V,a^5/V,ab/V,a^2b/V$ for $b',c'',
c,c'$ respectively, and get the lemma.
\end{proof}

Note that the  computations shows
\begin{align*}
a^6& =(y_1^2+y_2^2)^6v^{12}
=(y_1^{12}-y_1^{10}y_2^2+y_1^8y_2^4-y_1^6y_2^6+y_1^4y_2^8-y_1^2y_2^{10}
+y_2^{12})V^2\\
&=(y_1^{12}-y_1^6y_2^6+y_2^{12})V^2 =C^2V^2=D_2^2,
\end{align*}
where we use the fact $y_1^7y_2-y_1y_2^7=0$.   

\begin{lemma} \label{lem:lem7.2}
$ H^*(BE)^{3SD_{16}}\cong CA\{1,a,a^2,a^3/V,a^4/V,a^5/V\}.$
\end{lemma} 
\begin{proof}
Take the matrix 
$k'=\bigl(\begin{smallmatrix}-1&1\\-1&-1\end{smallmatrix}\bigr)$ such that
${\langle}3D_8,k'{\rangle}\cong 3SD_{16}$.  Then we have
\begin{align*}
{k'}^*{\co} a & =(y_1^2+y_2^2)v^2\mapsto ((-y_1+y_2)^2+(-y_1-y_2)^2)(2v)^2=a,\\
b&=y_1^2y_2^2v^4\mapsto (y_1^2-y_2^2)^2(2v)^4=2(a^2-4b)=2a^2-b.
\end{align*}
(If we take $\tilde b=b-a^2$, then ${k'}^*{\colon}\tilde b\mapsto 
-\tilde b$.)
Similarly we can compute $k' \co  d\mapsto -d$.
Then the lemma is almost immediate from the preceding lemma.
\end{proof}

\begin{lemma} \label{lem:lem7.3}
$ H^*(BE)^{3SD_{32}}\cong CA\{1,a^2,a^4/V\}$.
\end{lemma}
\begin{proof}
Take the matrix 
$l'=\left(\begin{smallmatrix}-1&3\\-3&-1\end{smallmatrix}\right)$ so that
${l'}^2=k'$ and ${\langle}3SD_8,l'{\rangle}\cong 3SD_{32}$. 
We see that
\[{l'}^*{\colon}
a=(y_1^2+y_2^2)v^2\mapsto ((-y_1+3y_2)^2+(-3y_1-y_2)^2)(3v)^2=-a,\]
which shows the lemma.
\end{proof} 

\begin{thm} \label{thm:thm7.4}
There is the isomorphism with $C'=C-a^3/V$
\[H^*(BO'N)\cong DA\{1,a,a^2, b,ab,a^2b\}
\oplus CA\{d,ad,a^2d,C',C'a,C'a^2\}\] 
\end{thm}
\begin{proof} 
Let $G=O'N$.
The orbits of $N_G(E)$--action of $A$--subgroups in $E$ are given by
$\{A_0,A_{\infty}\}$,$\{A_1,A_6\}$ and $\{A_2,A_3,A_4,A_5\}$.
From Ruiz and Viruel \cite{R-V}, $A_0$, $A_{\infty}$, $A_1$ and $A_6$ are
$F^{\ec}$--radical subgroups.  Hence we know that
\[H^*(O'N)\cong H^*(BE)^{3D_8}\cap 
i_{A_0}^{*-1}H^*(BA_0)^{SL_2(\bF_7){\co}2}
\cap i_{A_1}^{*-1}H^*(BA_1)^{SL_2(\bF_7){\co}2}.\]
For element $x=d$ or $x=C'$, the restrictions are
$x|A_0=x|A_1=0$.
Hence we see that $CA\{x\}$  are contained in $H^*(BG)$.
We can take  $C',C'a,C'a^2$ instead of $a^3/V$,
$a^4/V$ and $a^5/V$ 
as the $CA$--module generators since 
$a^3/V=(C-C').$
Moreover we know $CA\{C',C'a,C'a^2\}\subset H^*(BG)$.

It is known that 
$\bZ/p[y,u]^{SL_p(\bF_p)}\cong \bZ/p[\tilde D_1,\tilde D_2']$
where $\tilde D_2'=y_1u^p-y_1^pu$ 
and $(\tilde D_2')^{p-1}=\tilde D_2$.
Hence we know $\bZ/7[y,u]^{SL_2(\bF_7){\co}2}\cong \bZ/7[\tilde 
D_1,
(\tilde D_2)^2].$

Since $y_1v|A=\tilde D'_2$ we see
$a|A_0=(\tilde D'_2)^2, a|A_1=2(\tilde D'_2)^2$.
Hence $a,a^2$ are in $H^*(BG)$.
The fact $b|A_0=0$ and $b|A_1=(\tilde D'_2)^4$, implies that 
$b\in H^*(BG)$.  Hence all $a^ib^j$ are also in $H^*(BG)$.
\end{proof}

Next we consider the group $G=O'N{\co}2$.  Its Weyl group $W_G(E)$
is isomorphic to $3SD_{16}$. 
So we have $H^*(B(O'N{\colon}2))\cong H^*(BO'N)\cap 
H^*(BE)^{3SD_{16}}.$

\begin{cor} \label{cor:cor7.5} $H^*(B(O'N{\colon}2))\cong  (DA\{1,a,a^2\}
\oplus CA\{C',C'a,C'a^2\}).$
\end{cor}
\begin{cor} \label{cor:cor7.6}
$H^*(BRV_2)\cong  DA\{1,a,a^2,a^3,a^4,a^5\}.$
\end{cor}
\begin{proof} 
Let $G=RV_2$.  Since $A_2$ is also $F^{\ec}$--radical and $W_G(A_2)=
SL_2(\bF_7){\co}2$.  Hence we have
\[H^*(BG)\cong H^*(BE)^{3SD_{16}}
\cap i_{A_2}^{*-1}H^*(BA_2)^{SL_2(\bF_7){\co}2}.\]
Hence we have the corollary of the theorem.
\end{proof}

Since $H^*(BRV_3)\cong H^*(BE)^{3SD_{32}}\cap H^*(BRV_2)$, we have 
the following corollary.
\begin{cor}\label{cor:cor7.7}
$H^*(BRV_3)\cong  DA\{1, a^2,a^4\}.$
\end{cor}

\fullref{cor:cor7.7} can also be proved in the following way.
\begin{proof}
Let $G=RV_3$.  Since there is just 
one $G$--conjugacy class of $A$--subgroups, 
by Quillen's theorem \cite{Q}, we know
\[H^*(BRV_3)\subset H^*(BA_0)^{SL_2(\bF_7){\co}2}\cong 
DA\{1,(\tilde D_2')^2,(\tilde D_2')^4\}\text{ with }(\tilde D_2')^6=\tilde 
D_2.\]
Note that $a^2|A_0=(\tilde D_2')^4,\ a^4|A_0=(\tilde D_2')^2 \tilde 
D_2$ 
and $D_2|A_0=\tilde D_2$.
The fact ${k'}^*{\co}a\mapsto -a$ implies that 
$DA\{a^2,a^4\}\subset
H^*(BG)$ but $DA\{a,a^3,a^5\}\cap H^*(BG)=0$.
\end{proof}

\fullref{cor:cor7.6} can also be proved in the following way.
\begin{proof}
Let $G=RV_2$.  Since there is just 
two $G$--conjugacy classes of $A$--subgroups, 
by Quillen's theorem \cite{Q}, we know
\[H^*(BRV_2)\subset  H^*(BA_0)^{SL_2(\bF_7){\co}2}
\times H^*(BA_2)^{SL_2(\bF_7){\co}2}\]
Since $a\in H^*(BRV_2)$, the map 
$i_0^*{\co} H^*(BRV_2)\to H^*(BA_0)^{SL_2(\bF_7){\co}2}$
is epimorphism. Take  $b'=b^2-2a^2b$ so that 
$b'|A_0=b'|A_1=0$. Hence 
$$\Ker i_{A_0}^* \supset DA\{b',b'a,C'V\}.$$
Moreover $b'|A_2=(\wbar D_2')^2\wbar D_2,
b'a|A_2=(\wbar D_2')^4\wbar D_2,
c'V|A_2=(\wbar D_2).$
Since $(\wbar D_2')^2$ itself is not in the image of $i_{A_2}^*$,
we get the isomorphism
\[H^*(BRV_2)\cong DA\{1,a,a^2\}\oplus DA\{c'V,b',b'a\}.\proved\]
\end{proof}

\section[Cohomology for B7(1+2) II]
{Cohomology for $B7_+^{1+2}$ II}
\label{sec:sec8}

In this section, we study cohomology of $He,Fi_{24},RV_1$.

First we consider the group $G=He$.  
The multiplicative generators of
$H^*(He)$  are still computed by Leary \cite{L1}.
We will study more detailed cohomology structures here.
The Weyl group is $W_G(He)\cong 3S_3$.

\begin{lemma} \label{lem:lem8.1} The invariant $H^*(BE)^{3S_3}$ is
isomorphic to
\[ CA\otimes \{\bZ/7\{1,\wbar b,\wbar b^2\}\{1,\wbar a,\wbar b^3/V\}
\oplus \bZ/7\{\wbar d\}\{1,\wbar a,\wbar b,\wbar b^2/V,\wbar b^3/V\}
\oplus \bZ/7\{\wbar a^2\}),\]
where $\wbar a=(y_1^3+y_2^3)$, 
$\wbar b=y_1y_2v^2$ and $\wbar d=(y_1^3-y_2^3)v^3$.
\end{lemma} 
\begin{proof}
The group  $3S_3\subset  GL_2(\bF_7)$
is generated by $T'=\{\diag(\lambda,\mu)|\lambda^3=\mu^3=1)$ and 
$w'=\left(\begin{smallmatrix} 0&1\\ 1&0\end{smallmatrix}\right)$.  
If $y_1^iy_2^jv^k$ is invariant under $T'$, then $i=j=-k \mmod(3)$.
When $i,j\le 6,k\le 5$ but $(i,j)\not =(6,6)$,
the invariant monomials have the following terms
\begin{eqnarray*}
\{1,\wbar c=\wbar b^3/V=y_1^3y_2^3\}\{1,v^3\}\{1,\wbar b=y_1y_2v^2,
\wbar b'=y_1^2y_2^2v\} \\
\{1,v^3\}\{y_1^3,y_1^6,y_1y_2^4v^2,y_1^2y_2^5v,y_1^3y_2^6\},
\end{eqnarray*}
and terms obtained by exchanging $y_1$ and $y_2$.  Recall that
$w' \co y_1\mapsto y_2,y_2\mapsto y_1$ and $v\to -v$. 
The following elements are invariant 
\begin{equation*}
\begin{array}{lll}
\wbar a\wbar b  =(y_1y_2^4+y_1^4y_2)v^2, &
\wbar a\wbar b^2 =(y_1^2y_2^5+y_1^5y_2^2)v^4, &
\wbar a\wbar c = (y_1^3y_2^6+y_1^6y_2^3),\\
\wbar b\wbar d  =(y_1y_2^4-y_1^4y_2)v^5, &
\wbar b^2\wbar d/V  =(y_1^2y_2^5-y_2^5y_2^2)v, &
\wbar c\wbar d =(y_1^3y_2^6-y_1^6y_2^3)v^3 \\
\wbar a\wbar d =(y_1^6-y_2^6)v^3, &
\wbar a\wbar b^3/V  =y_1^3y_2^6+y_1^6y_2^3. & \\
\end{array}
\end{equation*}
Thus we get the lemma from \eqref{eqn:eqn3.2}.
\end{proof}

\begin{lemma} \label{lem:lem8.2}
$ H^*(BE)^{6S_3}\cong CA\otimes
(\bZ/7\{1,\wbar b,\wbar b^2\}\{1,\wbar b^3/V\}\oplus \bZ/7\{\wbar d\wbar a,
\wbar a^2\}).$
\end{lemma}
\begin{proof}
We can think $6S_3={\langle}S_3,\diag(-1,-1){\rangle}$. The action 
$\diag(-1,-1)$
are given by 
$\wbar a\mapsto -\wbar a$,$\wbar b\mapsto 
\wbar b$, and $\wbar d\mapsto -\wbar d$.
From \fullref{lem:lem8.1}, we have the lemma.
\end{proof} 

\begin{lemma} \label{lem:lem8.3}
$H^*(BE)^{6^2{\co}2}\cong CA\{1,\wbar b^2,\wbar c'',\wbar b^4/V\}$ 
where $\wbar c''=\wbar a^2-2\wbar b^3/V-2C.$
\end{lemma}
\begin{proof}
We can think $6^2{\co}2={\langle}3S_6,\diag(3,1){\rangle}$. The 
action 
$\diag(3,1)$
are given by 
$\wbar a^2\mapsto \wbar a^2-4\wbar c,\ \wbar b\mapsto -\wbar b,\
\wbar c\mapsto -\wbar c,\ \wbar d\wbar a\mapsto -\wbar d\wbar a$.
For example $\wbar b=y_1y_2v^2 \mapsto (3y_1)y_2(3v)^2=-\wbar b$.
Moreover we have $\wbar c''=Y_1+Y_2-2C\mapsto \wbar c''$.          
Thus we have the lemma.
\end{proof}                   

\begin{thm} \label{thm:thm8.4}
Let $\wbar c'=C+\wbar a^3/V. $Then there is the isomorphism
\[H^*(BHe)\cong DA\{1,\wbar b,\wbar b^2,\wbar d,\wbar d\wbar b,\wbar d\wbar 
b^2\}
\oplus CA\{\{\wbar a,\wbar c'\}\{1,\wbar b,\wbar b^2,\wbar d\}, \wbar a^2,
\wbar a^2\wbar c'\}.\]
\end{thm}
\begin{proof} 
Let $G=He$.
The orbits of $N_G(E)$--action of $A$--subgroups in $E$ are given by
\[\{A_0,A_{\infty}\},\qua \{A_1,A_2,A_4\},\qua \{A_3,A_5,A_6\}.\]
Since $A_6$ is the $F^{\ec}$--radical (see Leary \cite{L2}), we have
\[H^*(BHe)\cong H^*(BE)^{3S_3}\cap 
i_{A_6}^{*-1}H^*(BA_6)^{SL_2(\bF_7)}.\]
For element $x=\wbar a$ or $x=C+\wbar c=C+y_1^3y_2^3$, the restrictions 
are
$x|A_6=0$, eg $\wbar a|A_6=(y^3+(-y)^3)=0.$
Hence we see that $CA \{x\}$
are contained in $H^*(BG)$.

Since  $\wbar b=y_1y_2v^2$, we see 
$\wbar b|A_0=-y^2v^2=-(\tilde D_2')^2$.
Similarly $\wbar d|A_6=2(\tilde D_2')^3$.
Thus we can compute $H^*(BHe)$.
\end{proof}

\begin{cor}\label{cor:8.5}  $H^*(B(He{\co}2))\cong
DA\{1,\wbar b,\wbar b^2\}\oplus CA\{\wbar c',\wbar c'\wbar b,\wbar c'\wbar 
b^2,
\wbar a^2,\wbar a\wbar d\}.$
\end{cor}
\begin{thm}\label{thm:thm8.6} There is the isomorphism
\[H^*(BFi_{24}')\cong DA\{1,\wbar b,\wbar b^2,\wbar a^2V,\wbar c'\wbar bV,
\wbar c'\wbar b^2V\}
\oplus CA\{\wbar c'',\wbar a\wbar d\}
\text{ where }\wbar c''=\wbar a^2-2\wbar c'.\]
\end{thm}
\begin{proof} 
Let $G=Fi_{24}'$.  Since $A_1$ is also $F^{\ec}$--radical and $W_G(A_1)=
SL_2(\bF_7){\co}2$.  Hence we have
\[H^*(BG)\cong H^*(B(He{\co}2))
\cap i_{A_1}^{*-1}H^*(BA_1)^{SL_2(\bF_7){\co}2}.\]
For the elements $x=\wbar a\wbar d, \wbar c''(=Y_1+Y_2-2C)$, we see
$x|A_1=x|A_6=0$. Hence these elements are in $H^*(BG)$.
Note that $\wbar b|A_1=(\tilde D_2')^2$ and $\wbar b\in H^*(BG)$.
We also know $\wbar a^2V|A_1=\tilde D_2$.
\end{proof}

Since $H^*(BFi_{24})\cong H^*(BFi_{24}')\cap H^*(BE)^{6^2{\co}2}$ 
and
$\wbar b^4=1/2(\wbar a^2-2C-\wbar c'')V$, we have the following
corollary.
\begin{cor} \label{cor:cor8.7}
$ H^*(BFi_{24})\cong (DA\{1,\wbar b^2,\wbar b^4\}
\oplus CA\{\wbar c''\}). $              
\end{cor}
For $G=RV_1$, The subgroup $A_0$ is also $F^{\ec}$--radical, we see
\[H^*(BRV_1)\cong H^*(BFi_{24})\cap i_0^{-1*}H^*(BA_0)^{GL_2(\bF_7)}\]
Hence we have the following corollary.
\begin{cor} \label{cor:cor8.8} 
$H^*(BRV_1)\cong DA\{1,\wbar b^2,\wbar b^4,D_2''\}$ 
with  $\wbar b^6=D_2^2+D_2''D_2$.
\end{cor}
\begin{proof}
Let $D_2''=\wbar c''V=\wbar c''(D_1-C^6\wbar c'')$. Then we have
\[\wbar b^6=Y_1Y_2V^2=(Y_1+Y_2-C)CV^2=(C+(Y_1+Y_2-2C)CV^2
=D_2^2+(\wbar c''V)D_2.\]
Thus the corollary is proved.
\end{proof}

\section[Stable splitting for B7(1+2)]
{Stable splitting for $B7_+^{1+2}$}
\label{sec:sec9}

Let $G$ be groups considered in the preceding two sections, eg
$O'N$,$O'N{\co}2$,\ldots,$RV_1$.  First consider the dominant 
summands $X_{q,k}$.
From \fullref{cor:cor4.6}, the dominant summands are only related to 
$H=W_G(E)$.
Recall the notation $X_{q,k}(H)$ in \fullref{lem:lem4.7}.                  
The module $X_{q,k}(H)$ is still given in the  preceding sections.    
              
From \fullref{lem:lem7.1}, \fullref{lem:lem7.2}, \fullref{lem:lem7.3},
\fullref{lem:lem8.1}, \fullref{lem:lem8.2} and \fullref{lem:lem8.3}
we have
\begin{align*}
H=3D_8 &;  X_{6,0}=\{a^3/V,a^2b/V\}, X_{4,4}=\{a^2,b\},  
X_{2,2}=\{a\}, \\
& \qua X_{4,1}=\{d\}, 
X_{6,3}=\{ad\} \\ 
H=3SD_{16}&; X_{6,0}=\{a^3/V\},  
X_{4,4}=\{a^2\},  X_{2,2}=\{a\} \\
H=3SD_{32}&; X_{4,4}=\{a^2\} \\
H=3S_3 &; X_{6,0}=\{\wbar b^3/V,\wbar a^2\},
X_{4,4}=\{\wbar b^2\}, X_{2,2}=\{\wbar b\}, \\
& \qua X_{6,3}=\{\wbar a\wbar d\},X_{3,0}=\{\wbar a\},
  X_{5,2}=\{\wbar a\wbar b\},  X_{3,3}=\{\wbar d\},
 X_{5,5}=\{\wbar d\wbar b\}\\
H=6S_3 &; X_{6,0}=\{\wbar b^3/V,\wbar a^2\},
  X_{2,2}=\{\wbar b\}, X_{4,4}=\{\wbar b^2\},  X_{6,3}=
\{\wbar a\wbar d\} \\
H=6^2{\colon}2 &;  X_{6,0}=\{\wbar a^2-2\wbar b^3/V\},  
X_{4,4}=\{\wbar b^2\}.
\end{align*} 
For example,  ignoring nondominant summands, we have the following 
diagram
\[\stackrel{X_{0,0}\vee X_{4,4}}{\longleftarrow}
B(E{\co}3SD_{32})\stackrel{X_{6,0}\vee X_{2,2}}{\longleftarrow} 
B(E{\co}3SD_{16})\stackrel{X_{6,0}\vee X_{4,4}\vee X_{4,1}\vee 
X_{6,3}}
{\longleftarrow} B(E{\co}3D_{8}).\]
From \fullref{cor:cor4.4}, the number $m(G,1)_k$ is given by 
$\rank_pH^{2k}(BG)$
for $k{\langle}p-1$ and  $\rank_pH^{2p-2}(G)$ for $k=0$. For example 
when
$G=E{\co}3S_3$,
\[ m(G,1)_0=3, m(G,1)_3=1,\qua m(G,1)_k=0\text{ for } k\not =0,\not =3.\]
\begin{lemma}\label{lem:lem9.1}
Let $G$ be one of the $O'N,O'N,\ldots,Fi_{24}',RV_1$.  Then
the number $m(G,1)_k$ for $L(1,k)$ is given by
\begin{align*}
m(G,1)_0 & = \begin{cases} 2\text{ for } G=He,He{\colon}2\\
1\text{ for }G=O'N,O'N{\colon}2,Fi_{24},Fi_{24}'
\end{cases} \\
m(G,1)_3 & = \begin{cases} 1\text{ for } G=He, \\
m(G,1)_k=0 \text{ otherwise}.
\end{cases}
\end{align*}
\end{lemma}
Now we consider the number $m(G,2)_k$ of the non dominant summand 
$L(2,k)$.
\begin{lemma}\label{lem:9.2}
The classifying spaces $BG$ for $G=O'N,O'N{\co}2$ have the non 
dominant summands $M(2)\vee L(2,2)\vee L(2,4)$.
\end{lemma}
\begin{proof}
We only consider the case  $G=O'N$, and  the case $O'N{\co}2$ is
almost the same.         
The non $F^{\ec}$--radical groups are $\{A_2,A_3,A_4,A_5\}$
(recall the proof of \fullref{thm:thm7.4}).
The group $W_G(E)=3D_8\cong 
{\langle}\diag(2,2),\diag(1,-1),w{\rangle}$. 
Hence the normalizer group is
\[N_G(A_2)=E{\co}{\langle}\diag(2,2),\diag(-1,-1){\rangle}.\]
Here note that $w,\diag(1,-1)$ are not in the normalizer,
eg $w{\co}{\langle}c,ab^2{\rangle} \to 
{\langle}c,a^2b^{-2}{\rangle}={\langle}c,ab^6{\rangle}$.  Since
$\diag(2,2){\co}ab^2 \mapsto (ab^2)^2$, $c\mapsto c^4$
and 
$\diag(-1,-1){\co}ab^2 \mapsto (ab^2)^{-1}$, \linebreak $c\mapsto c,$
the Weyl groups are
\[ W_G(A_2)\cong U{\co}{\langle}\diag(4,2),\diag(1,-1){\rangle}.\]
Let $W_1=U{\co}\diag{\langle}4,2{\rangle}$. For 
$v=\lambda y_1^{p-1}\in M_{p-1,k}$,
we have  $\wbar W_1v=\lambda y_2^{p-1}$ since $2^3=1$, 
from the argument in the proof of \fullref{lem:lem4.11}.  Moreover 
\[  \overline{ 
{\langle}\diag(1,-1){\rangle}}y_2^{p-1}=(1+(-1)^k)y_2^{p-1}, \]
implies that the $BG$ contains $L(2,k)$ if and only if $k$ even.
\end{proof}
\begin{lemma}\label{lem:lem9.3}
The classifying space $BHe$ (resp. 
$B(He{\co}2)$,$Fi_{24}'$,$Fi_{24}$)
contains the non  dominant summands 
\begin{align*}
2M(2)\vee L(2,2)\vee L(2,4)\vee L(2,3)\vee L(1,3)&\\
(\text{resp. } 2M(2)\vee L(2,2)\vee L(2,4),\ \ M(2),\ \ M(2))&.
\end{align*}
\end{lemma}  
\begin{proof}
First consider the case $G=He$.  The non $F^{\ec}$--radical group are
\[\{A_0,A_{\infty}\},\qua \{A_1,A_2,A_4\}.\]
The group $W_G(E)\cong 3S_3={\langle}\diag(2,1),w'{\rangle}$.  So we 
see 
$N_G(A_0)=E{\colon}{\langle}\diag(2,1){\rangle}$, and this implies 
$W_G(A_0)\cong U{\co}{\langle}\diag(2,2){\rangle}$.  The fact 
$4^k=0\mmod(7)$ implies
$k=3\mmod(6)$. Hence  $BG$ contains the summand 
\[ M(2)\vee L(2,3)\vee L(1,3)\]
which is induced from $BA_0$.

Next consider the summands induced  from $BA_1$.
The normalizer and Weyl
group are $N_G(A_1)=E{\co}{\langle}w'{\rangle}$ and 
$W_G(A_1)=U{\colon}{\langle}\diag(-1,1){\rangle}$
since $w'{\colon} ab\mapsto ab,c \mapsto -c$.  So we get 
\[M(2)\vee
L(2,2)\vee L(2,4)\]
which is induced  from $BA_1$.

For  $G=He{\colon}2$, we see  $\diag(-1,-1) \in W_G(E)$, this implies
that $\diag(-1,-1)\in N_G(A)$ and $\diag(1,-1)\in W_G(A_0)$. 
This means that the non dominant 
summand  induced from $BA_0$ is $M(2)$ but is not $L(2,3)$.
We also know $U{\co}\diag(1,-1)\in W_G(A_1)$ but 
the summand  induced from $BA_1$ are not changed.

For groups $Fi_{24}'$,$Fi_{24}$, the non $F^{\ec}$--radical groups
make just one $G$--conjugacy class $\{A_0,A_{\infty}\}$.
So $BG$ dose not contain the summands induced  from $BA_1$.
\end{proof}

\begin{thm}\label{thm:thm9.4}
When $p=7$, we have the following stable decompositions
of $BG$ so that 
$\stackrel{X_1}{\leftarrow} \cdots \stackrel{X_s}{\leftarrow}
G$ means that $BG\sim X_1\vee \cdots \vee X_s$
\[
\stackrel{X_{0,0}}{\longleftarrow}
\begin{cases} \stackrel{X_{4,4}}{\longleftarrow}

RV_3 \stackrel{X_{6,0}\vee X_{2,2}}{\longleftarrow}

RV_2 \stackrel{M(2)\vee L(2,2)\vee L(2,4)}{\longleftarrow}

O'N{\co}2 \stackrel{\stackrel{X_{6,0}\vee X_{4,4}}
{\vee X_{4,1}\vee X_{6,3}}}{\longleftarrow} O'N \\ 
\quad \\            
\stackrel{X_{6,0}\vee X_{4,4}}{\longleftarrow}

RV_1 \stackrel{M(2)}{\longleftarrow}

Fi_{24} \stackrel{X_{6,0}\vee X_{6,3}\vee X_{2,2}}{\longleftarrow}

Fi_{24}' \stackrel{M(2)\vee L(2,2)\vee L(2,4)}
{\longleftarrow} He{\co}2
\end{cases} \] 
\[\qquad \qquad \qquad \stackrel{
X_{3,0}\vee X_{5,2}\vee X_{3,3}
\vee X_{5,5}\vee L(2,3)\vee L(1,3)}{\longleftarrow}He.\]
\end{thm}

We write down the cohomology of stable summands.
At first we see that $H^*(X_{0,0})\cong
H^*(BRV_3)\cap H^*(BRV_1)\cong DA.$
Here note that elements 
$a^2-(y_1y_2)^2v^4$ in \fullref{sec:sec7}
and $\wbar b^2=y_1^2y_2^2v^4$ in \fullref{sec:sec8}
are not equivalent under the action in $GL_7(\bF_7)$ 
because $y_1^2+y_2^2$ is indecomposable in $\bZ/7[y_1,y_2]$.

From the cohomologies, $H^*(BRV_3)$ and $H^*(BRV_2)$, then
$H^*(X_{4,4})\cong DA\{a^2,a^4\}$ and 
$H^*(X_{6,0}\vee X_{2,2})\cong DA\{a,a^3,a^5\}$.

On the other hand, we know $H^*(X_{6,0})$
from the cohomology $H^*(BRV_1)$.  Thus we get the following
lemma.
\begin{lemma}\label{lem:lem9.5}
There are isomorphisms of cohomologies
\begin{align*}
H^*(X_{0,0})& \cong DA, H^*(X_{4,4})\cong DA\{a^2,a^4\} \\
H^*(X_{6,0})& \cong DA\{D_2\}\cong DA\{a^3\}, 
H^*(X_{2,2})\cong DA\{a,a^5\}.
\end{align*}
\end{lemma}
Let us write $M\{a\}=DA\{1,C,\ldots,C^{p-1}\}\{a\}$.
From the facts that $D_2=CV$, $D_1=C^p+V$ and
$D_2=C(D_1-C^p)=CD_1-C^{p+1}$, 
we have two decompositions
\[ CA\{a\}\cong DA\{1,C,\ldots,C^p\}\{a\}
\cong DA\{a\}\oplus M\{Ca\} \cong  M\{a\}\oplus DA\{Va\}.\]
From the cohomology of
$H^*(Fi_{24})$, we know the following lemma.
\begin{lemma}\label{lem:lem9.6}
$H^*(M(2))\cong M\{C\}.$
\end{lemma}
Comparing  the cohomology
$H^*(B(He{\co}2))\cong 
H^*(BFi_{24}')\oplus M\{\wbar a^2,\wbar c'\wbar b,
\wbar c'\wbar b^2\},$
we have the isomorphisms
\[H^*(M(2))\cong M\{\wbar a^2\}, H^*(L(2,2)\vee L(2,4))
\cong M\{\wbar c'\wbar b,\wbar c'\wbar b^2\}.\]
From $H^*(BFi_{24}')\cong 
H^*(BFi_{24})\oplus DA\{\wbar a^2V,
\wbar c'\wbar bV,\wbar c'\wbar b^2V\}\oplus CA\{\wbar a\wbar d\}$,
we also know that
\[H^*(X_{6,3})\cong CA\{\wbar a\wbar d\},
H^*(X_{6,0}\vee X_{2,2})\cong DA\{\wbar a^2V,\wbar c'\wbar bV,
\wbar c'\wbar b^2V\}.\]
We still get   $H^*(BFi_{24})\cong H^*(BRV_1)\oplus M\{\wbar c''\}$
and $H^*(M(2))\cong M\{\wbar c''\}.$

Next consider the cohomology of groups studied in \fullref{sec:sec7} 
eg $O'N$.
There is the isomorphism
\[H^*(BO'N)\cong H^*(BO'N{\co}2)\oplus
DA\{b,b^2,ab^2\}
\oplus CA\{d,da,da^2\}.\]
Indeed, we have 
\begin{align*}
H^*(X_{6,0}\vee X_{4,4}) & \cong
DA\{b,b^2,ab^2\}\cong DA\{a^2,a^3,a^4\} \\ 
H^{*}(X_{6,3}) & \cong CA\{da\} \\
H^*(X_{4,1}) & \cong CA\{d,da^2\}.
\end{align*}
We also have the isomorphism
$H^*(BO'N{\colon}2)\cong H^*(BRV_2)\oplus M\{C',C'a,C'a^2\}$
and $H^*(M(2)\vee L(2,2)\vee L(2,4))\cong M\{C',C'a,C'a^2\}$.

Recall that
\[H^*(BE)^{3SD_{32}}\cong CA\{1,a^2,a^4/V\}
\cong DA\{1,a^2,a^4\}\oplus M\{C,a^2C,a^4/V\},\]
in fact $H^*(M(2)\vee L(2,2)\vee L(2,4))\cong M\{C,a^2C,a^4/V\}$.

\section[The cohomology of M for p=13]
{The cohomology of $\bM$ for $p=13$}
\label{sec:sec10}

In this section, we consider the case $p=13$ and $G=\bM$
the Fisher--Griess Monster group.
It is know that $W_G(E)\cong 3\times 4S_4$.  
The $G$--conjugacy classes of $A$--subgroups are 
divided  two classes ;
one is $F^{\ec}$--radical and the other is
not.  The class of $F^{\ec}$--radical groups contains 
$6$ $E$--conjugacy classes (see Ruiz--Viruel \cite{R-V}). 
(The description of \cite[(4.1)]{T-Y} was not correct, and the 
description of $H^*(B\bM)$ in \cite[Theorem 6.6]{T-Y} was not
correct.)  The Weyl group $W_G(A)\cong SL_2(\bF_{13}).4$ for
each $F^{\ec}$--radical subgroup $A$.

Since $S_4\cong PGL_2(\bF_3)$ [S], we have the presentation of
\[S_4={\langle}x,y,z|x^3=y^3=z^2=(xy)^2= 1, zxz^{-1}=y{\rangle}.\]
(Take $x=u,y=u'$ in \fullref{lem:lem4.8}, and $z=w$ in \fullref{sec:sec5}.)
By arguments in the proof of Suzuki \cite[Chapter~3~(6.24)]{S}, we can take
elements $x,y,z$ in $GL_2(\bF_{13})$ by
\begin{equation}\label{eqn:eqn10.1}
x=\left(\begin{smallmatrix}
3&0\\ 0&9\\ \end{smallmatrix}\right) , \qua       
y=\left(\begin{smallmatrix}
5&-4\\ -2&7\\ \end{smallmatrix}\right), \qua
z=\left(\begin{smallmatrix}
2&2\\ 1&-2\\ \end{smallmatrix}\right),  
\end{equation}
so that we have
\[x^3=y^3=1,\ zxz^{-1}=y,\ (xy)^2=-1,\ z^2=\diag(6,6).\]
Hence we can identify
\begin{equation}\label{eqn:eqn10.2}
\quad 3\times 4S_4\cong {\langle}x,y,z{\rangle} \subset 
GL_2(\bF_{13}).
\end{equation}
It is almost immediate that $H^*(BE)^{{\langle}x{\rangle}}$
(resp. $H^*(BE)^{{\langle}-1{\rangle}}$) is multiplicatively
generated by $y_1y_2,y_1^3,y_2^3$ (resp.
$y_1y_2,y_1^2,y_2^2$) as a $\bZ/(13)[C,v]$--algebra.
Hence we can write
\begin{align}\label{eqn:eqn10.3}
H^*(BE)^{{\langle}x,-1{\rangle}}& \cong \bZ/(13)[C,v]
\bigl\{\{1,y_1y_2,\ldots,(y_1y_2)^5\}\{(y_1y_2)^6,y_1^6,y_2^6\},\\ \nonumber
& \qua y_1^{12},y_2^{12}, y_1^{12}y_2^6,y_1^6y_2^{12}\bigr\}.
\end{align}
For the invariant $H^*(BE)^{{\langle}y,-1{\rangle}}$, we get the 
similar 
result
exchanging $y_i$ to $(z^{-1})^*y_i$ since $zxz^{-1}=y$. 
Indeed $(z^{-1})^*{\co}H^*(BE)^{{\langle}x,-1{\rangle}}
\cong H^*(BE)^{{\langle}y,-1{\rangle}}$.  

To seek invariants, we recall the relation between the $A$--subgroups
and elements in $H^2(BE;\bZ/p)$.  For 
$0\not =y=\alpha y_1+\beta y_2\in H^2(BE;\bZ/p)$, let 
$A_y=A_{-(\alpha/\beta)}$ so that $y|A_y=0$.  This induces
the $1-1$ correspondence,
\[(H^2(BE;\bZ/p)-\{0\})/F_p^*
\leftrightarrow \{A_i|i\in F_p\cup\{\infty\}\},
\qua y \leftrightarrow A_y.\]
Considering the map
$g^{-1}A_i \stackrel{g}{\to} A_i \subset
E \stackrel {\beta^{-1}y}{\to} \bZ/p,$
we easily see $A_{g^*y}=g^{-1}A_y$.

For example, the order $3$ element $x$ induces the maps
\begin{align*}
x^* \co & y_1-y_2\mapsto 3y_1-9y_2 
\mapsto 9y_1-3y_2 \mapsto y_1-y_2 \\
x^{-1} \co & A_{y_1-y_2}={\langle}c,ab{\rangle}\to 
{\langle}c,a^9b^{3}{\rangle}
\to {\langle}c,a^3b^{9}{\rangle}\to {\langle}c,ab{\rangle}.
\end{align*}
In particular $A_1,A_9,A_3$ are in the same $x$--orbit of
$A$--subgroups.  Similarly the ${\langle}x{\rangle}$--conjugacy classes 
of $A$ is given 
\[  \{A_0\}, \{A_{\infty}\}, \{A_1,A_3,A_9\},
\{A_2,A_5,A_6\}, \{A_4,A_{10},A_{12}\}, \{A_7,A_8,A_{11}\}. \] 
The ${\langle}y{\rangle}$--conjugacy classes are just $\{zA_i\}$  for
${\langle}x{\rangle}$--conjugacy classes $\{A_i\}$.
\[  \{A_7=zA_0\},\{A_{12}\}, \{A_3,A_1,A_5\},
\{A_6,A_9,A_2\}, \{A_{11},A_8,A_{\infty}\},
\{A_0,A_{10},A_{4}\}. \]
Hence we have the ${\langle}x,y{\rangle}$--conjugacy classes
\[  C_1= \{A_1,A_2,A_3,A_5,A_6,A_9\},
C_2=\{A_0,A_4,A_{10},A_{12}\},
C_3=\{A_{\infty},A_7,A_8,A_{11}\}. \]
At last we note ${\langle}x,y,z{\rangle}$--conjugacy classes are 
two classes $C_1,C_2\cup C_3.$

Let us write the ${\langle}x{\rangle}$--invariant
\begin{align}\label{eqn:eqn10.4}
u_6 & =\Pi_{A_i\in C_1}(y_2-iy_1)=
(y_2-y_1)(y_2-2y_1)\cdots(y_2-9y_1) \\ \nonumber
& =y_2^6-9y_1^3y_2^3+8y_1^6.
\end{align}
Then $u_6$ is also invariant under 
$y^*$ because the ${\langle}x,y{\rangle}$--conjugacy class 
$C_1$ divides two ${\langle}y{\rangle}$--conjugacy classes
\[C_1=\{A_1,A_3,A_5\}\cup \{A_2,A_6,A_9\} \]
and the element $u_6$ is rewritten as
\[u_6=\lambda(\Pi_{i=0}^2y^{i*}(y_2-y_1)).
(\Pi_{i=0}^2y^{i*}(y_2-2y_1)) \text{ for } 
\lambda\not =0\in \bZ/(13).\]
We also note that $u_6|A_i=0$ if and only if $i\in C_1$.
Similarly the following elements are 
${\langle}x,y{\rangle}$--invariant,
\begin{align}
u_8 & =\Pi_{A_i\in C_2\cup C_3}(y_2-iy_1)
=y_1y_2(y_2^6+9y_1^3y_2^3+8y_1^6) \\ \nonumber
u_{12} & =\Pi_{A_i\in C_2}(y_2-iy_1)^3=(y_2^4+y_1^3y_2)^3 \\ \nonumber
&  =\lambda(\Pi_{i=0}^2x^{i*}y_2)(\Pi_{i=0}^2x^{i*}(y_2-4y_1))^3 \\ \nonumber
& =\lambda'(\Pi_{i=0}^2y^{i*}(y_2-12y_1))(\Pi_{i=0}^2y^{i*}y_2)^3v\\ \nonumber
u_{12}' & =\Pi_{A_i\in C_3}(y_2-iy_1)^3=(y_1y_2^3+8y_1^4)^3.
\end{align}
Of course $(u_{12}u_{12}')^{1/3}=u_8$ and $u_6u_8=0$.  Moreover direct
computation shows $u_6^2=u_{12}+5u_{12}'$.
\begin{lemma}\label{lem:lem10.1}
$H^*(BE)^{{\langle}x,y{\rangle}}\cong
\bZ/(13)[C,v]\{1,u_6,u_6^2,u_6^3,u_8,u_8^2,u_{12}\}.$
\end{lemma}
\begin{proof}
Recall \eqref{eqn:eqn10.3} to compute 
\[H^*(BE)^{{\langle}x,y{\rangle}}\cong 
H^*(BE)^{{\langle}x,-1{\rangle}}
\cap H^*(BE)^{{\langle}y,-1{\rangle}}.\]
Since $(z^{-1})^*(y_1y_2)^i\not =(y_1y_2)^i$ for $1\le  i \le p-2$,
from \eqref{eqn:eqn10.3} we know invariants of the 
lowest positive degree are of the form
\[u=\gamma y_2^6+\alpha y_2^3y_1^3+\beta y_1^6.\]
Then $u'=u-\gamma u_6$ is also invariant with $u'|A_{\infty}=0$.
Hence $u'|A_i=0$ for all $A_i\in C_3$.  Thus we know
$u'=\lambda y_1^2(u_{12}')^{1/3}.$
But this is not ${\langle}y{\rangle}$--invariant for $\lambda\not =0$, 
because
$(u')^3=\lambda^3 y_1^6u_{12}'$ is invariant, while $y_1^6$ is
not ${\langle}y{\rangle}$--invariant. Thus we know $u'=0$.

Any $16$--dimensional invariant is form of
\[u=y_1y_2(\gamma y_2^6+\alpha y_2^3y_1^3+\beta y_1^6).\]
Since $u|A_0=u|A_{\infty}=0$, we know $u|A_i=0$ for all
$A_i\in C_2\cup C_3$.  Hence we know 
\[u=\gamma u_{12}^{1/3}(u_{12}')^{1/3}=\gamma u_8.\]
By the similar arguments, we can prove the lemma
for degree $\le 24$.

For $24 <$degree$<48$, we only need consider the elements
$u'=0 \mmod(y_1y_2)$.
For example, $H^{18}(BE;\bZ/13)^{{\langle}x,-1{\rangle}}$ is 
generated by
\[\{(y_1y_2)^9, (y_1y_2)^3C, y_1^6C, y_2^6C, y_1^6y_2^{12},
y_1^{12}y_2^6\}.\]
But we can take off $y_1^6C{=}y_1^{18}$, $y_2^6C{=}y_2^{18}$
by $\lambda u_6^3+\mu Cu_6$ so that $u'{=}0 \mmod(y_1y_2)$.

Hence we can take $u'$ so that $u_8$ divides $u'$
from the arguments similar to the case of degree=$16$.
Let us write $u'=u''u_8$.  Then we can
write
\[u''=y_1^ky_2^k(\lambda_1y_1^6+\lambda_2y_1^3y_2^3)
+\lambda_3(y_1y_2)^{k-3}C,\]
taking off $\lambda y_1^ky_2^ku_6$ if necessary
since $u_6u_8=0$.  (Of course, for $k{<}3$, $\lambda_3=0$.)
Since $u_8|A_i\not =0$ and $u_6|A_i=0$ for $i\in C_1$,
we have
\[(u''-y^*u'')|A_i=0\text{ for }i\in C_1.\]
Since $y^*y_1=5y_1-4y_2$ and $y^*y_2=-2y_1+7y_2$, we have
\begin{align*}
(u''-y^*u'')|A_i
& =\lambda_1(i^k-(5-4i)^{6+k}(-2+7i)^k) \\
& \qua +\lambda_2(i^{k+3}-(5-4i)^{k+3}(-2+7i)^{k+3}) \\
& \qua 
+\lambda_3(i^{k-3}-(5-4i)^{k-3}(-2+7i)^{k-3}).
\end{align*}
We will prove that we can take all $\lambda_i=0$.
Let us write $U=u''-y^*u''$.  
We then have the following cases.
\begin{enumerate}
\item The case $k=0$, ie degree=$14$.
If we take $i=1$,
\[U|A_1=\lambda_1(1-1)+\lambda_2(1-1^35^3)=0.\]
So we have $\lambda_2=0$.  We also see $\lambda_1=0$ since
$U|A_3=\lambda_1(1-(5-12)^6)=2\lambda_1=0.$
\item The case $k=1$.  Since $y_1y_2u_6-u_8=-18y_1^4y_2^4$,
we can assume $\lambda_2=0$
taking off $\lambda u_8^2$ if necessary.  
We have also $\lambda_1=0$ from
$U|A_1=\lambda_1(1^1-1^75^1)=0.$
\item The case $k=2$. We get the  the result 
$U|A_1=2\lambda_1+4\lambda_2$, $U|A_3=5\lambda_1+5\lambda_2$.
\item The case $k=3$.  First considering $Cu_8$,
we may take $\lambda_3=0$.  The result is given by
$U|A_1=6\lambda_1+2\lambda_2$ and $U|A_2=7\lambda_1+9\lambda_2$.
\item The case $k=4$. The result follows from
\[U|A_1=6\lambda_2+9\lambda_3,
U|A_3=6\lambda_1+6\lambda_2+6\lambda_3,
U|A_5=2\lambda_1-4\lambda_2+6\lambda_3.\]
\end{enumerate}
Hence the lemma is proved.\end{proof}
Next consider the invariant under ${\langle}x,y,\diag(6,6){\rangle}$.
The action for $\diag(6,6)$ is given by
$y_1^iy_2^jv^k \mapsto 6^{i+j+2k}y_1^iy_2^jv^k$.
Hence the invariant property implies $i+j+2k=0 \mmod(12)$.
Thus $H^*(BE)^{{\langle}x,y,\diag(6,6){\rangle}}$ is  generated 
as a $CA$--algebra by 
\[ \{1,u_6v^3,u_8v^2,u_{12},u_{12}',v^6\}. \]

\begin{lemma}\label{lem:lem10.2}
The invariant 
$H^*(BE)^{3\times 4S_4}\cong H^*(BE)^{{\langle}x,y,z{\rangle}}$
is isomorphic to
\[ CA\{1,u_6v^3,
(u_6v^3)^2,(u_6v^3)^3,u_8v^8,(u_8v^8)^2/V,(u_{12}-5u_{12}')\}.\]
\end{lemma}
\begin{proof}
We only need compute $z^*$--action.  Since
\[3\times 4S_4\cong 
{\langle}x,y,\diag(6,6){\rangle}{\colon}{\langle}z{\rangle},\]
the $z^*$--action on $H^*(BE)^{{\langle}x,y,\diag(6,6){\rangle}}$ is an 
involution.
Let $u_6v^3=u_6(y_1,y_2)v^3$.  First note
$u_6|A_{\infty}=u_6(0,y)=y^6.$
On the other hand, its  $z^*$--action  is
\begin{align*}
z^*u_6v^3|A_{\infty}& =u_6(2y_1+2y_2,y_1-2y_2)(-6v)^3|A_{\infty}
=u_6(2y,-2y)(-6v)^3\\
& =((-2)^6-9(-2)^3(2)^3+8(2)^6)(-6)^3y^6v^3 \\
&=(1+9+8)8y^6v^3=y^6v^3.
\end{align*}
Hence we know $u_6v^3$ is invariant, while $u_6v^9$ is not.

Similarly we know 
\[u_8v^2|A_1=u_8(y,y)v^2=5y^8v^2,
\qua z^*u_8v^2|A_1=-5y^8v^2.\]
Hence $u_8v^8$ and $u_8^2v^4$ are invariant but $u_8v^2$ is not.

For the action $u_{12}$, we have
\begin{equation*}
\begin{array}{rclrclrclrcl}
u_{12}|A_0 &=&0, &u_{12}|A_{\infty}&=&y^{12}, &
u_{12}'|A_0&=&5y^{12}, & u_{12}'|A_{\infty}&=&0,\\
z^*u_{12}|A_0&=&y^{12}, &  z^*u_{12}|A_{\infty}&=&0, &
z^*u_{12}'|A_0&=&0, & z^*u_{12}'|A_{\infty}&=&5y^{12}.
\end{array}
\end{equation*}
Thus we get
$z^*u_{12}=(1/5)u_{12}',\ k^*u_{12}'=5u_{12}.$
Hence we know $u_{12}+(1/5)u_{12}'$ and 
$(u_4^3-(1/5)u_{12}')v^6=(u_6v^3)^2$
are invariants. Thus we can prove the lemma.
\end{proof}
\begin{thm}\label{thm:thm10.3} 
For $p=13$, the cohomology $H^*(B\bM)$ is
isomorphic to 
\[ DA\{1,u_8v^8,(u_8v^8)^2\}\oplus CA\{u_6v^3,
(u_6v^3)^2,(u_6v^3)^3,(u_{12}-5u_{12}'-3C)\}.\]
\end{thm}
\begin{proof}
Direct computation shows
\[ u_{12}-5u_{12}'=y_2^{12}-2y_2^9y_1^3+3y_2^3y_1^9+y_1^{12},\]
and hence $u_{12}-5u_{12}'-3C|A_1=0$,
indeed, the restriction   
is zero for each $A_i\in C_1$. 
The isomorphism
\[H^*(B\bM)\cong H^*(BE)^{3\times 4S_4}\cap i_{A_1}^{-*}
(H^*(BA_1)^{SL_4(\bF_{13}).4},\]
completes the proof. 
\end{proof}
The stable splitting is given by the following theorem.
\begin{thm}
We have the stable splitting
\begin{align*}
B\bM & \sim X_{0,0}\vee X_{12,0}\vee X_{12,6}
\vee X_{6,3}\vee X_{8,8}
\vee M(2),\\
B(E{\co}3\times 4S_4) & \sim B\bM\vee M(2)
\vee L(2,4)\vee L(2,8).
\end{align*}
\end{thm}
\begin{proof}
Let $H=E{\co}3\times 4S_4.$  Recall that
\[
X_{q,k}(H) =(S(A)^q\otimes v^k)\cap H^*(BH)
\qua 0\le q\le 12, 0\le k\le 11.\] 
We already know
\[X_{*,*}(H)=\bZ/(13)\{ 1,u_8v^8,u_6v^3,u_6^2v^6,
u_{12}-5u_{12}'\}.\]
Hence $BH$ has the dominant summands in the theorem.      

The normalizer groups of $A_0,A_1$ are given
\[N_H(A_0)=E{\co}{\langle}x,\diag(6,6){\rangle}, 
N_H(A_1)=E{\co}{\langle}\diag(6,6){\rangle}.\]
Hence the Weyl groups are
\[W_H(A_0)=U{\co}{\langle}\diag(1,3),\diag(6^2,6){\rangle},
W_H(A_1)=U{\co}{\langle}\diag(6^2,6){\rangle}.\]
From the arguments of \fullref{lem:lem4.11}, the non-dominant 
summands induced from $BA_1$ are
$M(2)\vee L(2,4)\vee L(2,8).$
We also know the non-dominant summands from $BA_0$ are
$M(2)$. This follows from
\[ \overline{{\langle}\diag(1,3){\rangle}}y_2^{p-1}=
\sum_{i=0}^2(3^i)^ky_2^{p-1}\quad for \ y_2^{p-1}\in M_{p-1,k}\]
and this is nonzero mod($13$) if and only if $k=0 \mmod(3)$.
\end{proof}
\begin{rem}
It is known $H^*(Th)\cong DA$ for $p=5$ in \cite{T-Y}.
Hence all cohomology $H^*(BG)$ for groups $G$ in
\fullref{thm:thm2.1} (4)--(7) are explicitly known. For (1)--(3), see also
Tezuka--Yagita \cite{T-Y}.
\end{rem}

\section[Nilpotent parts of H*(BG,Z(p))]
{Nilpotent parts of $H^*(BG;\bZ_{(p)})$}
\label{sec:sec11}

It is known that $p^2H^*(BE;\bZ)=0$ (see Tezuka--Yagita \cite{T-Y} and
Leary \cite{L2}) and
\[pH^{*{>}0}(BE;\bZ)\cong \bZ/p\{pv,pv^2,\ldots\}.\]
In particular $H^{\odd}(BE;\bZ)$ is all just $p$--torsion.  
There is a decomposition
\[H^{\even}(BE;\bZ)/p\cong H^*(BE)\oplus N \text{ with }
N=\bZ/p[V]\{b_1,\ldots,b_{p-3}\}\]
where $b_i=Cor_{A_0}^E(u^{i+1}), |b_i|=2i+2$. 
(Note for $p=3$,$N=0$.) 
The restriction images $b_i|A_j=0$ 
for all $j\in \bF_p \cup \infty$.  
For $g\in GL_2(\bF_p)$, the induced action 
is given by $g^*(b_i)=\det(g)^{i+1}b_i$ by 
the definition of $b_i$.

Note that 
\[2=|y_i|{<}|b_j|=2(j+1){<}|C|=2p-2{<}|v|=2p .\]
So $g^*(y_i)$ is given by \eqref{eqn:eqn3.4}
also in $H^*(BE;\bZ)$ and $g^*(v)=\det(g)v \mmod(p)$. 
Hence we can identify that
\[H^*(BE)^H=(H^{\even}(BE;\bZ)/(p,N))^H \subset
H^{\even}(BE;\bZ/p)^H.\]                
Let us write the reduction map by 
$q{\co}H^*(BE;\bZ)\to H^*(BE;\bZ/p).$
\begin{lemma}\label{lem:lem11.1}
Let $H\subset GL_2(\bF_p)$ and $(|H|,p)=1$.
If $x\in H^*(BE)^H$, then there is $x'\in H^*(BE;\bZ)^H$
such that $q(x')=x$.
\end{lemma}
\begin{proof}
Let  $x\in H^*(BE)^H$ and $G=E{\co}H$.  
Then we can think $x\in H^*(BE;\bZ/p)^H \cong 
H^*(BG;\bZ/p)$ and $\beta(x)=0$.
By the exact sequence 
\[H^{\even}(BG;\bZ_{(p)})\stackrel{q}{\to}
H^{\even}(BG;\bZ/p))\stackrel{\delta}{\to} H^{\odd}(BG;\bZ_{(p)}),
\]
we easily see that $x\in Image(q)$ 
since  $q\delta(x)=\beta(x)=0$
and $q|H^{\odd}(BG;\bZ_{(p)})$ is injective.
Since $H^*(BG;R)\cong H^*(BE;R)^H$ for $R=\bZ_{(p)}$ or
$\bZ/p$, we get the lemma.
\end{proof}

%$Proof \ of \ Theorem \ 3.1.$
\begin{proof}[Proof of \fullref{thm:thm3.1}]
From Tezuka--Yagita \cite[Theorem 4.3]{T-Y} and Broto--Levi--Oliver \cite{B-L-O},  
we have the isomorphism
\[H^*(BG;\bZ)_{(p)}\cong H^*(BE;\bZ)^{W_G(E)}
  \cap _{A{\co}F^{\ec}-\text{radical}}i_A^{*-1}H^*(BA;\bZ)^{W_G(A)}.\]
The theorem  is immediate from the above lemma and the fact that
$H^{\even{>}0}(BA;\bZ)\cong H^{*{>}0}(BA)$.
\end{proof}

Let us write $N(G)=H^*(BG;\bZ)\cap N$.  Then 
\[H^{\even}(BG;\bZ)/p \cong H^*(BG)\oplus N(G).\] 
The nilpotent parts $N(G)$ depends only on the group
$\ddet(G)=\{\det(g)|g\in W_G(E)\}\subset \bF_p^*$,
in fact, $N(G)=N^{W_G(E)}=N^{\ddet(G)}$ . 
\begin{lemma}\label{lem:lem11.2}
If $\ddet(G)\cong \bF_p^*$ (eg $G=O'N,He,\ldots,RV_3$ 
for $p=7$, or $G=\bM$ for $p=13$), then
\[N(G)\cong \bZ/p[V]\{b_iv^{p-2-i}|1\le i\le p-3\}.\]
\end{lemma}
\begin{lemma}\label{lem:lem11.3}
Let $G$ have a $7$--Sylow subgroup $E$.  Then, we have
\[N(G)=\begin{cases}
\bZ/7[V]\{b_1v^4,b_2v^3,b_3v^2,b_4v\} 
\text{ if }\ddet(G)=\bF_7^*
\\
\bZ/7[v^3]\{b_1v,b_2,b_3v^2,b_4v\} \text{ if }\ddet(G)\cong \bZ/3
\\
\bZ/7[v^2]\{b_1,b_2v,b_3,b_4v\} \text{ if }\ddet(G)\cong \bZ/2\\
\bZ/7[v]\{b_1,b_2,b_3,b_4\} \text{ if }\ddet(G)\cong \{1\}.
\end{cases}
\]
\end{lemma}

Now we consider the odd dimensional elements.
Recall that $$H^{\odd}(BA;\bZ)\cong \bZ/p[y_1,y_2]\{\alpha\},$$
where $\alpha=\beta(x_1x_2)\in H^*(BA;\bZ/p)\cong \bZ/p[y_1,y_2]\otimes
\Lambda(x_1,x_2)$ with $\beta(x_i)=y_i$.
Of course $g^*(\alpha)=\det(g )\alpha$ for $g\in \Out(A)$.
For example $H^{\odd}(B(A{\colon}Q_8))\cong 
H^*(B(A{\colon}Q_8))\{\alpha\}$
since $\ddet(A{\co}Q_8)=\{1\}$.

Recall the Milnor operation 
$Q_{i+1}=[P^{p^n}Q_i-Q_iP^{p^n}],Q_0=\beta$. It is known that 
\[Q_1(\alpha)=y_1^py_2-y_1y_2^p=\tilde D_2' \text{ with }
(\tilde D_2')^{p-1}=\tilde D_2.\]
The submodule of $H^*(X;\bZ_{(p)})$ generated by 
(just) $p$--torsion additive generators 
can be identified with $Q_0H^*(X;\bZ/p)$. 
Since $Q_iQ_0=-Q_0Q_i$, we can extend the map
\cite[page 377]{Y1}
\[Q_i{\co}Q_0H^*(X;\bZ/p) \stackrel{Q_i}{\to}
Q_0H^*(X;\bZ/p)\subset H^*(X;\bZ_{(p)}).\]
Since all elements in $H^{\odd}(BA;\bZ)$ are
(just) $p$--torsion, 
we can define the map   
$$Q_1{\co}H^{\odd}(BA;\bZ)\to H^{\even}(BA;\bZ)=H^{\even}(BA).$$
Moreover this map is injective.
\begin{lemma}[Yagita \cite{Y1}]
\label{lem:lem11.4}
Let $G$ have the $p$--Sylow subgroup $A=(\bZ/p)^2$.
Then 
\[Q_1\ {\co}\ H^{\odd}(BG;\bZ_{(p)}) \cong 
(H^{\even}(BG)\cap J(G)),\]
with $J(G)=\mathrm{Ideal}(y_1^py_2-y_1y_2^p)\subset H^{\even}(BA).$
\end{lemma}
\begin{cor} \label{cor:cor11.5} For $p=3$, there are isomorphisms
\begin{align*}
H^{\odd}(BA;\bZ)^{Z/8} & \cong 
S\{b,a'b,a,(a_1-a_2)\}\{\alpha\} \\                  
H^{\odd}(BA;\bZ)^{D_8} & \cong 
S\{1,a,a_1,a'\}\{b\alpha\} \\
H^{\odd}(BA;\bZ)^{SD_{16}} & \cong 
S\{1,a'\}\{b\alpha\}.
\end{align*}
\end{cor}
\begin{proof}
We only prove the case $G=A{\co}\bZ/8$ since the proof of
the other cases are similar.
Note in $\S 5$ the 
element $Q_1(\alpha)$ is written by $b$
and  $b^2=a_1a_2$.  Recall $S=\bZ/3[a_1+a_2,a_1a_2]$.
Hence we get
\begin{align*}
H^*(BA)^{{\langle}l{\rangle}}\cap J(G) & \cong 
S\{1,a',ab,(a_1-a_2)b\}\cap {\rm Ideal}(b)\ \\
& =S\{b^2,b^2a',ab,(a_1-a_2)b\} \\
& = S\{b,ba',a,(a_1-a_2)\}\{Q_1(\alpha)\}.
\end{align*}
The corollary follows.
\end{proof}

By Lewis, we can write  \cite{L2,T-Y}
\[H^{\odd}(BE;\bZ)\cong \bZ/p[y_1,y_2]/
(y_1\alpha_2-y_2\alpha_1,
y_1^p\alpha_2-y_2^p\alpha_1)\{\alpha_1,\alpha_2\},
\] where $|\alpha_i|=3$.
It is also known that $Q_1(\alpha_i)=y_iv$ and
$Q_1{\co}H^{\odd}(BE;\bZ_{(p)})\to H^{\even}(BE)\subset
H^{\even}(BE;\bZ)/p$
is injective \cite{Y1}. Using this we can prove the
following lemma.
\begin{lemma}[Yagita \cite{Y1}]
\label{lem:lem11.6}
Let $G$ have the $p$--Sylow subgroup 
$E$.  Then
\[Q_1 {\co}\ H^{\odd}(BG) \cong 
(H^{\even}(BG)\cap J(G))\]
with $J(G)={\rm Ideal}(y_iv)\subset H^{\even}(BE).$
\end{lemma}
From the above lemma we easily compute the odd
dimensional elements.
Note that 
\[D_2=CV\not \in J(E)\text{ but }
D_2^2=C^2V^2=(Y_1^2+Y_2^2-Y_1Y_2)V^2\in J(E).\]
Let us write $\alpha =(Y_1y_1^{p-2}\alpha _1
+Y_2y_2^{p-2}\alpha _2
-Y_1y_2^{p-2}\alpha _2)Vv^{p-2}$ so that $Q_1(\alpha)=D_2^2$.
\begin{cor} \label{cor:cor11.7}
$H^{\odd}(B^2F_4(2)';\bZ_{(3)})\cong DA\{\alpha,\alpha '\}$
with $\alpha '=(y_1\alpha _1+y_2\alpha _2)v.$
\end{cor}
\begin{proof}
Recall that $H^*(B^2F_4(2)')\cong DA\{1,(Y_1+Y_2)V\}$
from the remark of \fullref{prop:prop6.3}.  The 
result is easily obtained          
from  $Q_1(\alpha)=D_2^2,Q_1(\alpha')=(Y_1+Y_2)V$.
\end{proof}
\begin{cor}  \label{cor:cor11.8} There are isomorphisms
\begin{align*}
H^{\odd}(BRV_3;\bZ_{(7)}) & \cong DA\{a,a^3,a^5\}\{\alpha'\}\\
H^{\odd}(BRV_2;\bZ_{(7)}) & \cong DA\{1,a,\ldots,a^5\}\{\alpha 
'\},
\end{align*}                 
with $\alpha'=(y_1\alpha_1+y_2\alpha_2)v$.
\end{cor}
\begin{proof}
We can easily compute
\[Q_1(\alpha')=Q_1((y_1\alpha_1+y_2\alpha_2)v)
=(y_1Q_1(\alpha_1)+y_2Q_1(\alpha_2))v
=(y_1^2+y_2^2)v^2=a.\]
Recall that $H^*(BRV_3)\cong DA\{1,a^2,a^4\}$. We get
\[H^*(BRV_3)\cap{\rm Ideal}(y_iv)=DA\{D_2^2,a^2,a^4\}
=DA\{a^5,a,a^3\}(Q_1\alpha '),\]
and the corollary follows.
\end{proof}
\begin{cor}  \label{cor:cor11.9}
$H^{\odd}(BRV_1;\bZ_{(7)})\cong DA\{\wbar b,\wbar b^3,\wbar b^5\}
\{\alpha''\}
\oplus DA\{\alpha\}$
where $\alpha''=y_1v\alpha_2$.
\end{cor}
\begin{proof}
Recall \fullref{cor:cor8.8}.We have
$C\wbar c''=C(Y_1+Y_2-2C)=-Y_1^2-Y_2^2+2Y_1Y_2.$
Hence we can see  $Q_1(\alpha)=-D_2\wbar c''V$.
\end{proof}
\begin{cor}  \label{cor:cor11.10}
The cohomology $H^{\odd}(B\bM;\bZ_{(13)})$ is isomorphic to 
\[ DA\{\alpha,\alpha_8,(u_8v^8)\alpha_8\}
\oplus CA\{\alpha _6,
(u_6v^3)\alpha_6,(u_6v^3)^2\alpha_6,\alpha_{12}\},\]
where 
\begin{align*}
\alpha_8 &=y_2(y_2^6+9y_2^3y_1^3+8y_1^6)v^7\alpha_1 \\
\alpha_6 & =(y_2^5\alpha_2-9y_2^2y_1^3 \alpha_2+8y_1^5y\alpha_1)v^2 \\
\alpha_{12} & =C(y_2^{11}\alpha_2
-2y_2^8y_1^3\alpha_2+3y_2^2y_1^9\alpha_2+y_1^{11}\alpha_1)v^{11}
-3\alpha/V.
\end{align*}
\end{cor}
\begin{proof}
It is almost immediate  that
\[Q_1(\alpha_8)=u_8v^8,\ \ Q_1(\alpha_6)=u_6v^3,\ \
Q_1(\alpha_{12})=(u_{12}-5u_{12}'-3C)CV.\]
From \fullref{thm:thm10.3}, we get the corollary.
\end{proof}

\bibliographystyle{gtart}
\bibliography{link}
\end{document}